\documentclass[10pt]{article}
\usepackage{amsmath, amsthm, amssymb, amsfonts, mathrsfs, mathtools}
\usepackage{graphicx}
\usepackage{makeidx}
\usepackage[left=2cm,right=2cm,top=2cm,bottom=2cm]{geometry}
\usepackage{caption}
\usepackage{subcaption}
\usepackage{tikz}
\usetikzlibrary{decorations.pathreplacing}
\tikzstyle{vertex}=[auto=left,circle,draw=black,fill=white, inner sep=1.5]

\newtheorem{theorem}{Theorem}[section]

\newtheorem{conj}[theorem]{Conjecture}

\newtheorem{prop}[theorem]{Proposition}
\newtheorem{corollary}{Corollary}[theorem]
\newtheorem{lemm}{Lemma}[section]

\newtheorem{ex}{Example}[section] 
\newtheorem{prob}{Problem}[section]

\title{ Non-isomorphic Signatures on Some Generalised Petersen Graph}

\author{ Deepak\\
Department of Mathematics\\
Indian Institute of Technology Guwahati\\
Guwahati, India - 781039\\
Email: deepakmath55555@iitg.ac.in\\
\\
Bikash Bhattacharjya\\
Department of Mathematics\\
Indian Institute of Technology Guwahati\\
Guwahati, India - 781039\\
Email: b.bikash@iitg.ac.in
}

\begin{document}
\maketitle

\vspace{-0.3in}

\begin{center}{Abstract}\end{center}
In this paper we find the number of different signatures of $P(3,1), P(5,1)$ and $P(7,1)$ upto switching isomorphism, where $P(n, k)$ denotes the generalised Petersen graph, $2k < n$. We also count the number of non-isomorphic signatures on $P(2n+1,1)$ of size two for all $n \geq 1$, and we conjecture that any signature of $P(2n+1,1)$, upto switching, is of size at most $n+1$.

\noindent {\textbf{Keywords}: Signed graph, generalised Petersen graph, balance, switching, switching isomorphism. 
 
\section{Introduction}\label{intro}

Throughout the paper we consider simple graphs. For all the graph-theoretic terms that have not been defined but are used in the paper, see Bondy \cite{Bondy}. Harary was the first to introduced signed graph and balance \cite{F.Harary}. Harary \cite{Harary} used them to model social stress in small groups of people in social psychology. Subsequently, signed graphs have turned out to be valuable. The fundamental property of signed graphs is balance. A signed graph is balanced if all its cycles have positive sign product. The second basic property of signed graphs is switching equivalence. Switching is a way of turning one signature of a graph into another, without changing cycle signs. Many properties of signed graphs are unaltered by switching, the set of unbalanced cycles is a notable example. In \cite{Sivaraman}, the non-isomorphic signatures on the Heawood graph are studied. The author in \cite{T.Zaslavsky} determined the non-isomorphic signed Petersen graph, using the fact that the minimal signature on a cubic graph is a matching. Using the same technique, we find the number of non-isomorphic signatures on $P(3,1), P(5,1)$ and $P(7,1)$. We also determine the number of non-isomorphic signatures of size two in $P(2n+1,1)$ for all $n\geq1$. 

\section{Preliminaries}\label{prelim}

A \textit{signified graph} is a graph $G$ together with an assignment of $+$ or $-$ signs to its edges. If $\Sigma$ is the set of negative edges, then we denote the signified graph by $(G, \Sigma)$. The set $\Sigma$ is called the signature of $(G, \Sigma)$. Signature $\Sigma$ can also  be viewed as a function from $E(G)$ into $\{+1, -1\}$. A \textit{resigning (switching)} of a signified graph at a vertex $v$ is to change the sign of each edge incident to $v$. We say $(G,\Sigma_{2})$ is \textit{switching equivalent} to $(G,\Sigma_{1})$ if it is obtained from $(G,\Sigma_{1})$ by a sequence of switchings. Equivalently, we say that $(G,\Sigma_{2})$ is switching equivalent to $(G,\Sigma_{1})$ if there exists a function $f:V\rightarrow \{+1, -1\}$ such that \linebreak[4] $\Sigma_{2}(e)$ = $f(u)\Sigma_{1}(e)f(v)$ for each edge $e=uv$ of $G$. Resigning defines an equivalence relation on the set of all signified graphs over $G$ (also on the set of signatures). Each such class is called a \textit{signed graph} and is denoted by $[G,\Sigma]$, where $(G,\Sigma)$ is any member of the class.

We say two signified graphs $(G, \Sigma_{1})$ and $(H, \Sigma_{2})$ to be \textit{isomorphic} if there exists a graph isomorphism $\psi : V(G) \rightarrow V(H)$ which preserve the edge signs. We denote it by $\Sigma_{1} \cong \Sigma_{2}$. They are said to be switching isomorphic if $\Sigma_{1}$ is isomorphic to a switching of $\Sigma_{2}$.That is, there exists a representation $(H, \Sigma_{2}')$ which is equivalent to $(H, \Sigma_{2})$ such that $\Sigma_{1} \cong \Sigma_{2}'$. We denote it by $\Sigma_{1} \sim \Sigma_{2}$. 

\begin{prop} \cite{Naserasr} 
  If $G$ has $m$ edges, $n$ vertices and $c$ components, then there are $2^{(m-n+c)}$ distinct signed graphs of $G$.
\end{prop}

One of the first theorems in the theory of signed graphs tells that the set of unbalanced cycles uniquely determines the class of signed graphs to which a signified graph belongs. More precisely, we state the following theorem.

\begin{theorem}\cite{Zaslavsky}\label{Signature}
Two signatures $\Sigma_{1}$ and $\Sigma_{2}$ of a graph $G$ are equivalent if and only if they have the same set of unbalanced cycles.
\end{theorem}

\section{Notations}\label{notation}

The \textit{distance} between two vertices $x$ and $y$ in a graph $G$, denoted $d_{G}(x,y)$, is the length of a shortest path connecting $x$ and $y$. The \textit{distance} between two edges $e_{1}$ and $e_{1}$ of a graph $G$, denoted $d_{G}(e_{1},e_{2})$, is the number of vertices of a shortest path connecting their end points(vertices). For example, $d_{G}(e_{1},e_{2})=2$  for the edges $e_{1}$ = $u_{0}u_{1}$ and $e_{2}$ = $u_{2}u_{3}$ of the graph $P(5,1)$ in Figure~\ref{P5}. Throughout this paper, the solid lines and dotted lines in a graph represent positive and negative edges respectively.

In a signed graph $[G, \Sigma]$, a signature $\Sigma^{'}$ which is equivalent to $\Sigma$ is said to be a \textit{minimal signature} if the number of edges in $\Sigma^{'}$ is minimum among all equivalent signatures of $\Sigma$. We denote the number of edges in $\Sigma^{'}$ by $|\Sigma^{'}|$. For example, if $[G, \Sigma]$ is balanced then $\Sigma^{'}$ = $\emptyset$ and thus $|\Sigma^{'}| = 0$. Notice that there may be two or more than two minimal signatures for a signed graph $[G, \Sigma]$. For example, for the signed graph $[K_3, \Sigma]$, where $\Sigma$ = $\{12, 23, 31\}$, the equivalent signatures $\Sigma_1 = \{12\}$ and $\Sigma_2 = \{23\}$ are minimal signatures of $[K_3, \Sigma]$. This shows that minimality of a signature is not unique. The following theorem tells about the maximum degree of a vertex in a minimal signature, when the signature is considered as a spanning subgraph of a given graph $G$.

\begin{theorem}\label{Min}
Let $[G, \Sigma]$ be a signed graph on $n$ vertices and let $\Sigma^{\prime}$ be an equivalent minimal signature of $\Sigma$. Then $d_{G_{\Sigma^{\prime}}}(v)$ $\leq$ $\lfloor \frac{n-1}{2} \rfloor$ for each vertex $v \in V(G_{\Sigma^{\prime}})$.
\end{theorem}
\begin{proof}
Let, if possible, there exists a vertex $u \in V(G_{\Sigma})$ such that $d_{G_{\Sigma}}(u) > \frac{n-1}{2}$. Resign at $u$ to get an equivalent signature $\Sigma_1$. It is clear that $|\Sigma| > |\Sigma_1|$. We apply the same operation on $\Sigma_1$, if $G_{\Sigma_1}$ has a vertex of degree greater than $\frac{n-1}{2}$. Repeated application, if needed, of this process will ultimately give us an equivalent signature $\tilde{\Sigma}$ of minimum number of edges such that degree of every vertex of $\tilde{\Sigma}$ is at most $\lfloor \frac{n-1}{2} \rfloor$. It is clear that $|\tilde{\Sigma}| = |\Sigma^{\prime}|$, and every vertex of $\Sigma^{\prime}$ have degree at most $\lfloor \frac{n-1}{2} \rfloor$.
\end{proof}

The following theorem will remain our key result throughout this paper.

\begin{theorem}\cite{T.Zaslavsky}\label{matching}
Every minimal signature of a cubic graph is a matching.
\end{theorem}

From now onward, matching of a graph $G$ stands for a minimal signature. With a few exceptions, most of the time switching transforms a matching (when considered as a signature) of $P(n, 1)$ to a new matching. The notation $\Sigma(e_1,e_2,\ldots,e_k)$ denotes a signature or a set of edges $\Sigma$ which contains the edges $e_1, e_2,\ldots,e_k$ of a graph. For example, in the graph $P(3,1)$ of Figure~\ref{p1}, $\Sigma(u_0u_1, v_1v_2)$ denotes a signature containing the edges $u_0u_1$ and $v_1v_2$.

Further, we say that two signatures $\Sigma_1$ and $\Sigma_2$ of a graph $G$ are \textit{automorphic} if there exists an automorphism $f$ of $G$ such that $uv \in \Sigma_1$ if and only if $f(u)f(v) \in \Sigma_2$. If two signatures are automorphic then they are said to be \textit{automorphic type} signatures. If two signatures $\Sigma_1$ and $\Sigma_2$ of a graph $G$ are not automorphic to each other, then we say that they are distinct automorphic type signatures. For example, in the signed graphs $[P(5,1), \{u_1u_2\}]$ and $[P(5,1), \{u_3u_4\}]$, the signatures $\{u_1u_2\}$ and $\{u_3u_5\}$ are automorphic type signatures.

\section{Generalised Petersen Graph} \label{gen-petersen}

Let $n$ and $k$ be positive integers such that $2 \leq 2k < n$. The generalized Petersen graph, denoted by $P(n,k)$, is defined to have the vertex set $\{u_{0},u_{1},\ldots,u_{n-1},v_{0},v_{1},\ldots,v_{n-1}\}$ and edge set 
$$\{u_{i}u_{i+1}~:~i = 0,1,\ldots,n-1\} \cup \{v_{i}v_{i+k}~:~i = 0,1,\ldots,n-1\} \cup \{u_{i}v_{i}~:~i = 0,1,\ldots,n-1 \},$$
where the subscripts are read modulo $n$. We call the cycle $u_0u_1\ldots u_{n-1}u_0$ as outer cycle and the cycle $v_0v_{k}v_{2k}\ldots v_0$ as inner cycle of $P(n,k)$. The edges of the form $u_iv_i$ are called the spokes of $P(n,k)$. It is clear that $P(2n+1,1)$ has $4n+2$ vertices and $6n+3$ edges. Now we discuss certain structural facts of $P(2n+1,1)$, where $n \geq 1$. 
 
\begin{theorem}\label{cycle}
For $n \geq 1$ and $2 \leq l \leq 2n+1$, the number of $2l$-cycles and the number of $(2n+1)$-cycles of $P(2n+1,1)$ are $2n+1$ and $2$, respectively. 
\end{theorem}
\begin{proof}
It is obvious that the cycles given by $\{u_{0}, u_{1},\ldots,u_{2n}\}$ and $\{v_{0}, v_{1},\ldots,v_{2n}\}$ are the only cycles of length $2n+1$. This proves the second part of the theorem.

We prove the first part of the theorem by counting the number of $2l$-cycles, where $2 \leq l \leq (2n+1)$. It is important to note that any even cycle in $P(2n+1,1)$ must contain as many $u_{i}$'s as $v_{i}$'s. It is clear that for each $i = 0, 1, \cdots ,2n$, the cycle $u_{i}v_{i}v_{i+1}\ldots v_{i+(l-1)}u_{i+(l-1)}u_{i+(l-2)}\ldots u_{i+1}u_{i}$ is of length $2l$, and any even  cycle of $P(2n+1,1)$ is of this form. Hence there are $2n+1$ cycles of $P(2n+1,1)$ of length $2l$, where  $2 \leq l \leq (2n+1)$. This proves the theorem. 
\end{proof}

\begin{theorem}\label{distance}
Distance between any two edges in $P(2n+1,1)$ is at most $n+1$ for all $n \geq 1$.
\end{theorem}
\begin{proof}
It is clear that the distance between any two spokes of $P(2n+1,1)$ is at most $n+1$. Further, distance between any two edges of the outer, as well as of the inner cycle, is at most $n+1$. Without loss of generality, if we pick the edge $u_0u_1$ from the outer cycle, then the edges $v_{n}v_{n+1}$ and $v_{n+1}v_{n+2}$ are the only edges of the inner cycle which are at maximum distance of $n+1$ from $u_0u_1$. Similarly, $u_{n+1}v_{n+1}$ is the only spoke which is at maximum distance of $n+1$ from $u_0u_1$. This completes the proof of the theorem. 
\end{proof}

For each $k = 0, 1, 2,\ldots,2n$, we define the permutations $\gamma, \rho_k, \delta_k$ of $V(G)$ such that for all \linebreak[4] $i = 0,1,2,\ldots,2n$, we have
\[\gamma(u_{i}) = v_{i}, \gamma(v_{i}) = u_{i} \text{ and }\rho_k(u_{i}) = u_{i+k}, \rho_k(v_{i}) = v_{i+k} ; \]
\[    \delta_k(u_i)=\left\{
                \begin{array}{ll}
                  u_i  \quad \text{if }\hspace{0.2cm} i=k, \\
                  u_l \quad \text{if } d(u_i,u_k) =    
                  d(u_l,u_k) \text{ and } i \neq k, i \neq l;
                \end{array}
              \right.\]
 \[   \delta_k(v_i)=\left\{
                \begin{array}{ll}
                  v_i  \quad \text{if }\hspace{0.2cm} i=k,  \\
                  v_l \quad \text{if } d(v_i,v_k) =    
                  d(v_l,v_k)  \text{ and }   i \neq k, i \neq l.
                \end{array}
              \right.  \]  
Note that each $\rho_k$ represents a clockwise rotation of $P(2n+1,1)$. Also each $\delta_k$ represents a reflection of $P(2n+1,1)$ about a line induced by the edge $u_kv_k$. Further, $\gamma$ just swaps the inner and outer cycles of $P(2n+1,1)$. Thus the automorphism group of $P(2n+1,1)$ is given by 
$$\text{Aut}(P(2n+1,1)) =\left\langle  \rho_k, \hspace{0.1cm} \delta_k, \hspace{0.1cm} \gamma \hspace{0.1cm} | \hspace{0.1cm} k = 0,1,2,\ldots,2n \right\rangle. $$ 
Accordingly, the automorphisms of  $P(2n+1,1)$ are some combination of rotations, reflections and interchanges of $v_{i}$'s with $u_{i}$'s. Using this fact, if $H_{1}$ and $H_{2}$ are two given subgraphs of $P(2n+1,1)$, it is easier to decide whether there is an automorphism of $P(2n+1,1)$ that maps $H_{1}$ onto $H_{2}$.

\begin{ex}\label{ex1}
The graph $P(3,1)$ is given in Figure~\ref{p1}. The automorphism $\rho_1$ rotates the graph $P(3,1)$ clockwise through the angle $\frac{2\pi}{3}$, the automorphism $\delta_1$ flips $P(3,1)$ about the line containing the edge $u_1v_1$ as its segment, and $\gamma$ switches the cycles $u_0u_1u_2$ and $v_0v_1v_2$
to each other.
\end{ex}
For more on automorphism group of generalised Petersen graph, see \cite{Frucht}.

\section{Signings on $P(3,1)$}\label{P(3,1)}

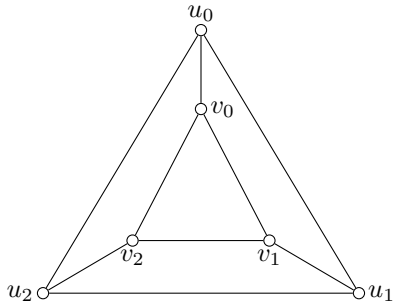
\begin{figure}[ht]
\centering
\begin{tikzpicture}[scale=0.7]
\node[vertex] (v1) at (10,0) {};\node[vertex] (v2) at (16,0) {};\node[vertex] (v3) at (13,5) {};
\node[vertex] (v4) at (11.7,1) {};\node[vertex] (v5) at (14.3,1) {};\node[vertex] (v6) at (13,3.5) {};

\node [left] at (v1) {$u_{2}$};
\node [right] at (v2) {$u_{1}$};
\node [above] at (v3) {$u_{0}$};
\node [below] at (v4) {$v_{2}$};
\node [below] at (v5) {$v_{1}$};
\node [right] at (v6) {$v_{0}$};

\foreach \from/\to in {v1/v2,v2/v3,v3/v1,v4/v5,v5/v6,v6/v4,v1/v4,v2/v5,v3/v6} \draw (\from) -- (\to);

\end{tikzpicture}
\caption{The graph $P(3,1)$.}\label{p1}
\end{figure}

\noindent
From Theorem~\ref{matching}, it is easy to see that finding non-isomorphic signatures on $P(3,1)$ is equivalent to determine the non-isomorphic matchings of $P(3,1)$ of size upto three. Let $M_k$ denotes a matching of size $k$, where $k = 0,1,2,3$. We classify all the automorphic type matchings of $P(3,1)$ of size upto three in the following lemmas. Let a  matching of size zero be denoted by $\Sigma_0$. 

\begin{lemm} \label{MS3.1}
The number of distinct automorphic type matchings of $P(3,1)$ of size one is two.
\end{lemm}
\begin{proof}
A matching of size one that does not contain a spoke is $\Sigma_{1} (u_0u_1)$. A matching of size one containing a spoke is $\Sigma_{2} (u_0v_0)$. It is easy to see that any other matching of size one is automorphic to either $\Sigma_1$ or $\Sigma_2$, and that $\Sigma_1$ is not automorphic to $\Sigma_2$. This proves the lemma.
\end{proof}

\begin{lemm} \label{MS3.2}
The number of distinct automorphic type matchings of $P(3,1)$ of size two is four.
\end{lemm}
\begin{proof}
We classify the matchings of size two by looking at the distance between their edges. Theorem~\ref{distance} gives us that distance between any two edges of $P(3,1)$ is at most two.
\begin{enumerate}

\item[(i)] Let $M_{2}$ have no spoke. We may assume that one edge is $u_{0}u_{1}$. There are two possibilities for such matchings of size two. One of such matchings is $\Sigma_3 (u_{0}u_{1}, v_{0}v_{1})$ and another is $\Sigma_{4} (u_{0}u_{1}, v_{1}v_{2})$.

\item[(ii)] Let $M_{2}$ have one spoke and let it be $u_0v_0$. One of such matchings is $\Sigma_{5} (u_{0}v_{o}, u_{1}u_{2})$.

\item[(iii)] Let $M_{2}$ have two spokes. One of such matchings is $\Sigma_{6}(u_{0}v_{0}, u_{1}v_{1})$.
\end{enumerate}
Any other matching of $P(3,1)$ of size two is automorphic to $\Sigma_{3}, \Sigma_{4}, \Sigma_{5}$ or $\Sigma_{6}$. Further, no two of these matchings are automorphic. This concludes the proof of the lemma.
\end{proof}

\begin{lemm} \label{MS3.3}
The number of distinct automorphic type matchings of $P(3,1)$ of size three is two.
\end{lemm}
\begin{proof}
Any $M_3$ must contain at least one spoke, as at most one edge can be taken from the inner cycle as well as from the outer cycle. If a matching of size three contains two spokes of $P(3,1)$, then no other edge can be included in that matching. Thus, following are the possibilities for $M_3$.
\begin{enumerate}

\item[(i)] Let $M_{3}$ have one spoke and let it be $u_0v_0$. There is only one possibility for such a matching, and let it be $\Sigma_{7} (u_{0}v_{0}, u_{1}u_{2}, v_{1}v_{2})$. 

\item[(ii)] Let $M_{3}$ has three spokes and let that $M_3$ be $\Sigma_{8} (u_{0}v_{0}, u_{1}v_{1}, u_{2}v_{2})$.
\end{enumerate}
Any other matching of size three is automorphic to $\Sigma_{7}$ or $\Sigma_{8}$, and that $\Sigma_{7}$ is not automorphic to $\Sigma_{8}$. This completes the proof.
\end{proof}

The matchings obtained in the preceding lemmas along with $\Sigma_{0}$ give us nine different automorphic type matchings of $P(3,1)$ \textit{viz.,} $\Sigma_{0}, \Sigma_{1},\ldots,\Sigma_{8}$. However, some of these nine matchings may be switching isomorphic to each other. We have the following observations.

\begin{itemize}
\item In $\Sigma_{6}$, by resigning at $u_{0}, u_{1}, u_{2}$; we get a matching automorphic to $\Sigma_{2}$. Thus $\Sigma_{6} \sim \Sigma_{2}$.
\item In $\Sigma_{7}$, by resigning at $u_{1}, v_{1}, v_{0}$;  we get a matching automorphic to $\Sigma_{4}$. Thus $\Sigma_{7} \sim \Sigma_{4}$.
\item In $\Sigma_{8}$, by resigning at $u_{0}, u_{1}, u_{2}$; we get a matching automorphic to $\Sigma_{0}$. Thus $\Sigma_{8} \sim \Sigma_{0}$.
\end{itemize}

So we are left with the matchings $\Sigma_{0}, \Sigma_{1}, \Sigma_{2}, \Sigma_{3}, \Sigma_{4}, \Sigma_{5}$, and their corresponding signed graphs are depicted in Figure \ref{DS3}, where the label of the vertices correspond to that of Figure \ref{p1}. In the following theorem we show that these six matchings are not switching isomorphic to each other.

\begin{theorem}
There are exactly six signed $P(3,1)$ upto switching isomorphisms.
\end{theorem}

\begin{proof}
The number of negative 3-cycles and negative $4$-cycles for the signed $P(3,1)$ shown in Figure~\ref{DS3} are given in Table~\ref{table1}.
\begin{table}[ht]
\begin{center}
\begin{tabular}{ |c|c|c|c|c|c|c| } 
 \hline
  & $\Sigma_{0}$ & $\Sigma_{1}$ & $\Sigma_{2}$ & $\Sigma_{3}$ & $\Sigma_{4}$ & $\Sigma_{5}$\\ 
 \hline
Number of negative $C_{3}$   & 0 & 1 & 0 & 2 & 2 & 1 \\
 \hline 
Number of negative $C_{4}$   & 0 & 1 & 2 & 0 & 2 & 3 \\ 
 \hline
\end{tabular}
\end{center}
\caption{Number of negative 3-cycles and negative $4$-cycles of some signed $P(3,1)$.}
\label{table1}
\end{table}
We see that the set of unbalanced cycles are different for all these six signatures. So by Theorem \ref{Signature}, we conclude that all these six signatures are pairwise not switching isomorphic. This completes the proof.
\end{proof}

\begin{figure}[ht]
\begin{minipage}{0.3\textwidth}
\begin{tikzpicture}[scale=0.45]
\node[vertex] (v1) at (10,0) {};\node[vertex] (v2) at (16,0) {};\node[vertex] (v3) at (13,5) {};
\node[vertex] (v4) at (11.7,1) {};\node[vertex] (v5) at (14.3,1) {};\node[vertex] (v6) at (13,3.5) {};
\node [below] at (13,-0.2) {$\Sigma_{0}$};

\foreach \from/\to in {v1/v2,v2/v3,v3/v1,v4/v5,v5/v6,v6/v4,v1/v4,v2/v5,v3/v6} \draw (\from) -- (\to);
\end{tikzpicture}
\end{minipage}
\hfill
\begin{minipage}{0.3\textwidth}
\begin{tikzpicture}[scale=0.45]
\node[vertex] (v1) at (10,0) {};\node[vertex] (v2) at (16,0) {};\node[vertex] (v3) at (13,5) {};
\node[vertex] (v4) at (11.7,1) {};\node[vertex] (v5) at (14.3,1) {};\node[vertex] (v6) at (13,3.5) {};
\node [below] at (13,-0.2) {$\Sigma_{1}$};

\foreach \from/\to in {v1/v2,v3/v1,v4/v5,v5/v6,v6/v4,v1/v4,v2/v5,v3/v6} \draw (\from) -- (\to);
\draw [dotted] (v2) -- (v3);

\end{tikzpicture}
\end{minipage}
\hfill
\begin{minipage}{0.3\textwidth}
\begin{tikzpicture}[scale=0.45]
\node[vertex] (v1) at (10,0) {};\node[vertex] (v2) at (16,0) {};\node[vertex] (v3) at (13,5) {};
\node[vertex] (v4) at (11.7,1) {};\node[vertex] (v5) at (14.3,1) {};\node[vertex] (v6) at (13,3.5) {};
\node [below] at (13,-0.2) {$\Sigma_{2}$};

\foreach \from/\to in {v1/v2,v2/v3,v3/v1,v4/v5,v5/v6,v6/v4,v1/v4,v2/v5} \draw (\from) -- (\to);
\draw [dotted] (v3) -- (v6);

\end{tikzpicture}
\end{minipage}
\end{figure}

\begin{figure}[ht]
\begin{minipage}{0.3\textwidth}
\begin{tikzpicture}[scale=0.45]
\node[vertex] (v1) at (10,0) {};\node[vertex] (v2) at (16,0) {};\node[vertex] (v3) at (13,5) {};
\node[vertex] (v4) at (11.7,1) {};\node[vertex] (v5) at (14.3,1) {};\node[vertex] (v6) at (13,3.5) {};
\node [below] at (13,-0.2) {$\Sigma_{3}$};

\foreach \from/\to in {v1/v2,v3/v1,v4/v5,v6/v4,v1/v4,v2/v5,v3/v6} \draw (\from) -- (\to);
\draw [dotted] (v2) -- (v3);\draw [dotted] (v5) -- (v6);

\end{tikzpicture}
\end{minipage}
\hfill
\begin{minipage}{0.3\textwidth}
\begin{tikzpicture}[scale=0.45]
\node[vertex] (v1) at (10,0) {};\node[vertex] (v2) at (16,0) {};\node[vertex] (v3) at (13,5) {};
\node[vertex] (v4) at (11.7,1) {};\node[vertex] (v5) at (14.3,1) {};\node[vertex] (v6) at (13,3.5) {};
\node [below] at (13,-0.2) {$\Sigma_{4}$};

\foreach \from/\to in {v1/v2,v3/v1,v4/v6,v5/v6,v1/v4,v2/v5,v3/v6} \draw (\from) -- (\to);
\draw [dotted] (v2) -- (v3);\draw [dotted] (v5) -- (v4);

\end{tikzpicture}
\end{minipage}
\hfill
\begin{minipage}{0.3\textwidth}
\begin{tikzpicture}[scale=0.45]
\node[vertex] (v1) at (10,0) {};\node[vertex] (v2) at (16,0) {};\node[vertex] (v3) at (13,5) {};
\node[vertex] (v4) at (11.7,1) {};\node[vertex] (v5) at (14.3,1) {};\node[vertex] (v6) at (13,3.5) {};
\node [below] at (13,-0.2) {$\Sigma_{5}$};

\foreach \from/\to in {v1/v4,v3/v1,v4/v5,v5/v6,v6/v4,v2/v5,v3/v2} \draw (\from) -- (\to);
\draw [dotted] (v2) -- (v1);\draw [dotted] (v3) -- (v6);

\end{tikzpicture}
\end{minipage}
\caption{The six signed $P(3,1)$.}
\label{DS3}
\end{figure}
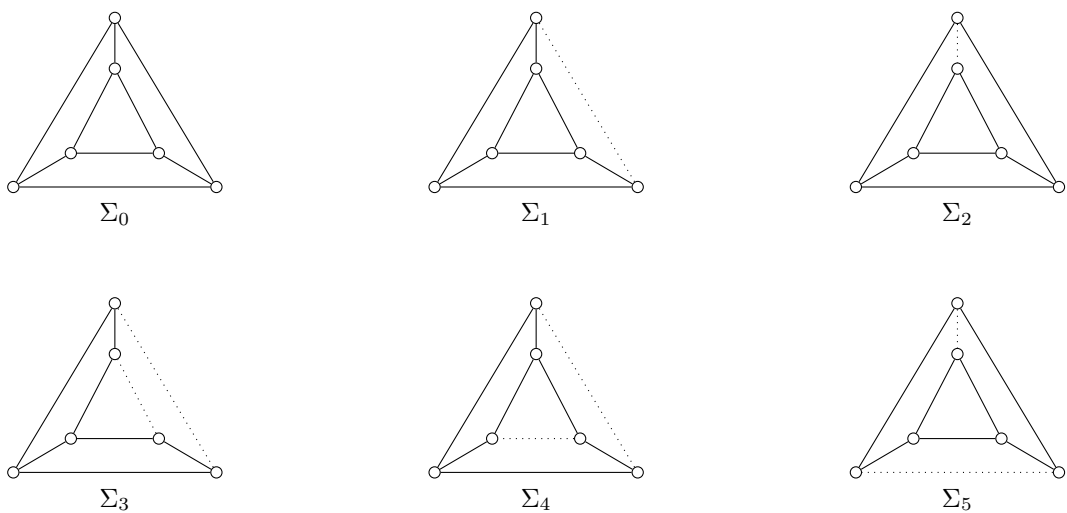

\section{Signings on $P(5,1)$}\label{P(5,1)}

The graph $P(5,1)$ is shown in Figure~\ref{P5}.
\begin{figure}[ht]
\centering
\begin{tikzpicture}[scale=0.5]

\node[vertex] (v1) at (12,6) {};\node[vertex] (v2) at (15.5,3.5) {};\node[vertex] (v3) at (14,0) {};\node[vertex] (v4) at (10,0) {};\node[vertex] (v5) at (8.5,3.5) {};
\node[vertex] (v6) at (12,4.5) {};\node[vertex] (v7) at (13.8,3.25) {};\node[vertex] (v8) at (13.0,1.4) {};\node[vertex] (v9) at (11,1.4) {};\node[vertex] (v10) at (10.3,3.25) {};

\node [above] at (v1) {$u_{0}$};\node [right] at (v2) {$u_{1}$};\node [right] at (v3) {$u_{2}$};\node [left] at (v4) {$u_{3}$};\node [left] at (v5) {$u_{4}$};
\node [below] at (v6) {$v_{0}$};\node [left] at (v7) {$v_{1}$};\node at (12.7,1.8) {$v_{2}$};\node at (11.3,1.8) {$v_{3}$};\node [right] at (v10) {$v_{4}$};

\foreach \from/\to in {v1/v2,v2/v3,v3/v4,v4/v5,v5/v1,v6/v7,v7/v8,v8/v9,v9/v10,v10/v6,v1/v6,v2/v7,v3/v8,v4/v9,v5/v10} \draw (\from) -- (\to);

\end{tikzpicture}
\caption{The graph $P(5,1)$.}\label{P5}
\end{figure}
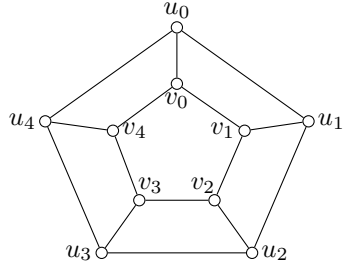
Recall from Theorem~\ref{matching} that finding non-isomorphic signatures of $P(5,1)$ is equivalent to finding matchings of $P(5,1)$ of sizes $0, 1, 2, 3, 4$ and $5$, upto switching isomorphism. We now classify all the automorphic type matchings of $P(5,1)$ of sizes upto five. We denote a matching of size zero by $\Sigma_{0}$. We emphasize that at most two edges of a matching may lie on the outer cycle or on the inner cycle. We use this fact to get the possible automorphic type matchings of different sizes. 

\begin{lemm} \label{MS5.1}
The number of distinct automorphic type matchings of $P(5,1)$ of size one is two.
\end{lemm}
\begin{proof}We have only the following two cases.
\begin{enumerate}
\item[(i)] Let $M_{1}$ have no spoke. There is only one automorphic type matching of size one. One of such matchings is $\Sigma_{1}(u_0u_1)$.

\item[(ii)] Let $M_{1}$ have one spoke.  There is also only one automorphic type matching of size one. One of such matchings is $\Sigma_{2}(u_{0}v_{0})$. 
\end{enumerate}
Any other matching of $P(5,1)$ of size one is automorphic to $\Sigma_{1}$ or $\Sigma_{2}$, and that $\Sigma_{1}$ is not automorphic to $\Sigma_{2}$. This completes the proof.
\end{proof}

\begin{lemm} \label{MS5.2}
The number of distinct automorphic type matchings of $P(5,1)$ of size two is eight.
\end{lemm}
\begin{proof}
We classify the matchings of size two by looking at the distance between the edges of the matching.
\begin{enumerate}
\item[(i)] Let the edges of the matching be at distance two. There are five different automorphic type matchings of size two and one of each such automorphic type matchings is $\Sigma_{3}(u_{0}u_{1}, v_{0}v_{1}), \Sigma_{4}(u_{0}u_{1}, v_{1}v_{2})$, $ \Sigma_{5}(u_{0}u_{1}, u_{2}u_{3}), \Sigma_{6}(u_{0}u_{1}, v_{2}u_{2})$ and $\Sigma_{7}(u_{0}v_{0}, u_{1}v_{1})$. Let $M_2$ be  a matching of size two other than $\Sigma_3, \Sigma_4, \Sigma_5, \Sigma_6$ and $\Sigma_7$ whose edges are at distance two. Note that $M_2$ must contain either two spokes, or two edges from outer cycle, or two edges from inner cycle, or one edge from outer cycle and one from inner cycle, or one edge from outer/inner cycle and one spoke. In each of these cases, $M_2$ is automorphic to either $\Sigma_7, \Sigma_5, \Sigma_3, \Sigma_4$ or $\Sigma_6$. Thus $\Sigma_3, \Sigma_4, \Sigma_5, \Sigma_6$ and $\Sigma_7$ are the only automorphic type matchings of size two whose edges are at distance two. It is clear that these matchings are pairwise non-automorphic.

\item[(ii)] Let edges of $M_{2}$ be at distance three. There are three automorphic type matchings of size two whose edges are at distance three. We denote them by $\Sigma_{8}(u_{0}u_{1}, v_{2}v_{3}), \Sigma_{9}(u_{0}u_{1}, v_{3}u_{3})$ and $\Sigma_{10}(u_{0}v_{0}, u_{2}v_{2})$. In a similar manner (as in case(i)), one can show that any other matching of size two whose edges are at distance three is automorphic to one of $\Sigma_{8},\Sigma_{9}$ and $\Sigma_{10}$. The matchings $\Sigma_{8}, \Sigma_{9}$ and $\Sigma_{10}$ are clearly pairwise non-automorphic. 
\end{enumerate}
 This concludes the proof of the lemma.
\end{proof}

\begin{lemm} \label{MS5.3}
The number of distinct automorphic type matchings of $P(5,1)$ of size three is $11$.
\end{lemm}
\begin{proof}
We classify all the automorphic type matchings of size three by looking at the number of spokes contained in these matchings.
\begin{enumerate}

\item[(i)] Let $M_{3}$ have no spoke. Out of three edges of $M_3$, two edges lie on outer (inner) cycle and the remaining one edge lies on inner (outer) cycle. Because of the automorphism $\gamma$, we may assume that two edges are lying on the outer cycle, and let they be $u_{0}u_{1}$ and $u_{2}u_{3}$. Therefore the possible automorphic type matchings for this case are $\Sigma_{11}(u_{0}u_{1}, v_{0}v_{1}, u_{2}u_{3}), \Sigma_{12}(u_{0}u_{1}, v_{1}v_{2}, u_{2}u_{3})$ and $\Sigma_{13}(u_{0}u_{1}, v_{3}v_{4}, u_{2}u_{3})$. Any other matching of size three which does not contain a spoke is automorphic to one of $\Sigma_{11}, \Sigma_{12}$ and $\Sigma_{13}$. Further, these matchings are not automorphic to each other.

\item[(ii)] Let $M_{3}$ have one spoke and let it be $u_0v_0$. If the other two edges of $M_3$ lie either on the outer cycle or on the inner cycle, then one of such matchings is $\Sigma_{14}(u_{0}v_{0}, u_{1}u_{2}, u_{3}u_{4})$. If one edge of $M_3$ lies on the outer cycle and one lies on the inner cycle, then following are the only possibilities: $\Sigma_{15}(u_{0}v_{0},u_{1}u_{2}, v_{1}v_{2}), \Sigma_{16}(u_{0}v_{0}, u_{1}u_{2}, v_{2}v_{3})$ and $\Sigma_{17}(u_{0}v_{0}, u_{2}u_{3}, v_{2}v_{3})$. Any other matching of size three containing only one spoke is automorphic to one of $\Sigma_{14}, \Sigma_{15}, \Sigma_{16}$ and $\Sigma_{17}$. Further, no two of these matchings are automorphic.

\item[(iii)] Let $M_{3}$ have two spokes. If the spokes are consecutive then there is only one possibility, \textit{viz.,}\linebreak[4] $\Sigma_{18}(u_{0}v_{0}, u_{1}v_{1}, u_{2}u_{3})$. If the spokes are not consecutive then there is also only one possibility, \textit{viz.}, $\Sigma_{19}(u_{0}v_{0}, v_{2}u_{2}, u_{3}u_{4})$. Any other matching of size three containing only two spokes is automorphic to $\Sigma_{18}$ or $\Sigma_{19}$, and that $\Sigma_{18}$ is not automorphic to $\Sigma_{19}$.

\item[(iv)] Let $M_{3}$ have three spokes. In this case, there are only two automorphic type matchings of size three and one of each such type of matchings is $\Sigma_{20}(u_{0}v_{0}, u_{1}v_{1}, u_{2}v_{2})$ and $\Sigma_{21}(u_{0}v_{0}, u_{1}v_{1}, u_{3}v_{3})$. Any other matching of size three containing only spokes is automorphic $\Sigma_{20}$ or $\Sigma_{21}$.
\end{enumerate}
This proves the lemma.
\end{proof}

\begin{lemm} \label{MS5.4}
The number of distinct automorphic type matchings of $P(5,1)$ of size four is $10$.
\end{lemm}
\begin{proof}
We classify the matchings of size four by considering the number of spokes contained in these matchings.
\begin{enumerate}
\item[(i)] Let $M_4$ have no spoke. Note that, at most two edges of $M_4$ may lie on the outer cycle and at most two edges may lie on the inner cycle. So, without loss of generality, let the edges $u_0u_1$ and $u_2u_3$ be lie on the outer cycle. Thus following are the only possibilities for the matchings of size four having no spoke: $\Sigma_{22}(u_0u_1, u_2u_3, v_0v_1, v_2v_3), \Sigma_{23}(u_0u_1, u_2u_3, v_0v_1, v_3v_4)$ and $\Sigma_{24}(u_0u_1, u_2u_3, v_1v_2, v_3v_4)$. Any other matching of size four which does not contain a spoke is automorphic to one of $\Sigma_{22}, \Sigma_{23}$ and $\Sigma_{24}$. Further, these matchings are pairwise non-automorphic. 

\item[(ii)] Let $M_4$ have one spoke and let it be $u_0v_0$. Out of the three remaining edges, two edges will lie on the outer (inner) cycle and one edge will lie on inner (outer) cycle. The two edges which lie on the outer cycle can be taken to be $u_1u_2$ and $u_3u_4$. Thus the possible automorphic type matchings of size four are $\Sigma_{25}(u_0v_0, u_1u_2, u_3u_4, v_1v_2)$ and $\Sigma_{26}(u_0v_0, u_1u_2, u_3u_4, v_3v_2)$. Any other matching of size four having only one spoke is automorphic to either $\Sigma_{25}$ or $\Sigma_{26}$, and that $\Sigma_{25}$ is not automorphic to $\Sigma_{26}$.

\item[(iii)] Let $M_4$ have two spokes. If spokes are at distance two then let they be $u_0v_0$ and $u_1v_1$. Further, out of remaining two edges, only one edge may lie on outer cycle and other may lie on inner inner. Let $u_2u_3$ lies on outer cycle. Then, the possible automorphic type matchings are $\Sigma_{27}(u_{0}v_{0}, u_{1}v_{1}, u_{2}u_{3}, v_{2}v_{3})$ and $\Sigma_{28}(u_{0}v_{0}, u_{1}v_{1}, u_{2}u_{3}, v_{3}v_{4})$. If the two spokes are at distance three then let they be $u_0v_0$ and $u_2v_2$. The only possibility for such matching is $\Sigma_{29}(u_{0}v_{0}, u_{2}v_{2}, u_{3}u_{4}, v_{3}v_{4})$. Any other matching of size four with only two spokes is automorphic to one of $\Sigma_{27}, \Sigma_{28}$ and $\Sigma_{29}$. Also these matchings are pairwise non-automorphic.

\item[(iv)] Let $M_4$ have three spokes. If one of the spokes is at distance three from the other two spokes, then no edge from the outer or inner cycle can be contained in $M_4$. Therefore the only possibility is $\Sigma_{30}(u_{0}v_{0}, u_{1}v_{1}, u_{2}v_{2}, u_{3}u_{4})$.

\item[(v)] Let $M_4$ have four spokes. One of such matchings is $\Sigma_{31}(u_{0}v_{0}, u_{1}v_{1}, u_{2}v_{2}, u_{3}v_{3})$. Any other matching for this case is automorphic to $\Sigma_{31}$.
\end{enumerate}
This completes the proof of the lemma.
\end{proof}

\begin{lemm} \label{MS5.5}
The number of distinct automorphic type matchings of $P(5,1)$ of size five is three.
\end{lemm}
\begin{proof}
It is clear that a matching $M_5$ of size five must have at least one spoke. Further, if $M_5$ has exactly two spokes, then only three vertices are unsaturated in the outer cycle as well as in the inner cycle. However, we must have at least two edges in $M_5$ either from the outer cycle or from the inner cycle. Therefore $M_5$ cannot have exactly two spokes. Similarly, $M_5$ cannot have four spokes. Thus following are the only possible cases.
\begin{enumerate}
\item[(i)] Let $M_{5}$ have one spoke. There is only one such automorphism type $M_5$. We denote it by \linebreak[4] $\Sigma_{32}(u_{0}v_{0}, u_{1}u_{2}, v_{1}v_{2}, u_{3}u_{4}, v_{3}v_{4})$.

\item[(ii)] Let $M_{5}$ have three spokes. There is only one automorphism type of $M_5$. We denote it by $\Sigma_{33}(u_{0}v_{0}, u_{1}v_{1}, u_{2}v_{2}, u_{3}u_{4}, v_{3}v_{4})$.

\item[(iii)] Let $M_{5}$ have five spokes. There is also only one automorphism type of $M_5$. We denote it by $\Sigma_{34}(u_{0}v_{0}, u_{1}v_{1}, u_{2}v_{2}, u_{3}v_{3}, u_{4}v_{4})$.
\end{enumerate}
This completes the proof of lemma.
\end{proof}
The matchings obtained in the preceding lemmas along with $\Sigma_{0}$ give us $35$ different automorphic type matchings of $P(5,1)$ \textit{viz.,} $\Sigma_{0}, \Sigma_{1},\ldots,\Sigma_{34}$. However, some of these $35$ matchings may be switching isomorphic to each other. We have the following observations.

\begin{itemize}
\item In $\Sigma_{7} $, by resigning at $u_{0}, u_{1}$; we get a matching automorphic to $\Sigma_{5}$. Thus $\Sigma_{7} \sim \Sigma_{5}$.
\item In $\Sigma_{11}$, by resigning at $u_{1}, u_{2}, v_{1}, v_{2}$; we get a matching automorphic to $\Sigma_{1}$. Thus $\Sigma_{11} \sim \Sigma_{1}$.
\item In $\Sigma_{12}$, by resigning at $u_{1}, u_{2}, v_{1}$; we get a matching automorphic to $\Sigma_{6}$. Thus $\Sigma_{12} \sim \Sigma_{6}$.
\item In $\Sigma_{13}$, by resigning at $u_{1}, v_{1}, u_{2}, v_{2}, v_{3}$; we get a matching automorphic to $\Sigma_{9}$. Thus $\Sigma_{13} \sim \Sigma_{9}$.
\item In $\Sigma_{14}$, by resigning at $u_{0}, u_{1}, u_{4}$; we get a matching automorphic to $\Sigma_{10}$. Thus $\Sigma_{14} \sim \Sigma_{10}$.
\item In $\Sigma_{15}$, by resigning at $u_{0}, u_{1}, v_{1}$; we get a matching automorphic to $\Sigma_{4}$. Thus $\Sigma_{15} \sim \Sigma_{4}$.
\item In $\Sigma_{17}$, by resigning at $u_{0}, u_{1}, u_{2}, v_{1}, v_{2}$; we get a matching automorphic to $\Sigma_{4}$. Thus $\Sigma_{17} \sim \Sigma_{4}$.
\item In $\Sigma_{18}$, by resigning at $u_{1}, u_{2}, u_{0}$; we get a matching automorphic to $\Sigma_{9}$. Thus $\Sigma_{18} \sim \Sigma_{9}$.
\item In $\Sigma_{20}$, by resigning at $u_{0}, u_{1}, u_{2}$; we get a matching automorphic to $\Sigma_{5}$. Thus $\Sigma_{20} \sim \Sigma_{5}$.
\item In $\Sigma_{21}$, by resigning at $u_{0}, u_{1}, u_{2},u_{3}, u_{4}$; we get a matching automorphic to $\Sigma_{10}$. Thus $\Sigma_{21} \sim \Sigma_{10}$.
\item In $\Sigma_{22}$, by resigning at $u_{1}, v_{1}, u_{2}, v_{2}$; we get a matching automorphic to $\Sigma_{0}$. Thus $\Sigma_{22} \sim \Sigma_{0}$.
\item In $\Sigma_{23}$, by resigning at $u_{1}, u_{2}, v_{1}, v_{2}, v_{3}$; we get a matching automorphic to $\Sigma_{2}$. Thus $\Sigma_{23} \sim \Sigma_{2}$.
\item In $\Sigma_{24}$, by resigning at $u_{1}, u_{2}, v_{2}, v_{3}$; we get a matching automorphic to $\Sigma_{6}$. Thus $\Sigma_{24} \sim \Sigma_{6}$.
\item In $\Sigma_{25}$, by resigning at $u_{0}, u_{1}, u_{4}, v_{0}, v_{1}, v_{4}$; we get a matching automorphic to $\Sigma_{6}$. Thus $\Sigma_{25} \sim \Sigma_{6}$.
\item In $\Sigma_{26}$, by resigning at $u_{0}, u_{1}, u_{4}$; we get a matching automorphic to $\Sigma_{19}$. Thus $\Sigma_{26} \sim \Sigma_{19}$.
\item In $\Sigma_{27}$, by resigning at $u_{0}, u_{1}, u_{2}, v_{2}$; we get a matching automorphic to $\Sigma_{8}$. Thus $\Sigma_{27} \sim \Sigma_{8}$.
\item In $\Sigma_{28}$, by resigning at $u_{0}, u_{1}, u_{2}$; we get a matching automorphic to $\Sigma_{16}$. Thus $\Sigma_{28} \sim \Sigma_{16}$.
\item In $\Sigma_{29}$, by resigning at $u_{3}, v_{2}, v_{3}$; we get a matching automorphic to $\Sigma_{16}$. Thus $\Sigma_{29} \sim \Sigma_{16}$.
\item In $\Sigma_{30}$, by resigning at $u_{0}, u_{1}, u_{2}, u_{4}$; we get a matching automorphic to $\Sigma_{6}$. Thus $\Sigma_{30} \sim \Sigma_{6}$.
\item In $\Sigma_{31}$, by resigning at $u_{0}, u_{1}, u_{2}, u_{3}, u_{4}$; we get a matching automorphic to $\Sigma_{2}$. Thus $\Sigma_{31} \sim \Sigma_{2}$.
\item In $\Sigma_{32}$, by resigning at $v_{2}, u_{2}, u_{3}, v_{3}$; we get a matching automorphic to $\Sigma_{2}$. Thus $\Sigma_{32} \sim \Sigma_{2}$.
\item In $\Sigma_{33}$, by resigning at $u_{0}, u_{1}, u_{2}, u_{3}, v_{3}$; we get a matching automorphic to $\Sigma_{8}$. Thus $\Sigma_{33} \sim \Sigma_{8}$.
\item In $\Sigma_{34}$, by resigning at $u_{0}, u_{1}, u_{2}, u_{3}, u_{4}$; we get a matching automorphic to $\Sigma_{0}$. Thus $\Sigma_{34} \sim \Sigma_{0}$.
\end{itemize}
Thus we are left with $12$ different matchings \textit{viz.,} $\Sigma_{0}, \Sigma_{1}, \Sigma_{2}, \Sigma_{3}, \Sigma_{4}, \Sigma_{5}, \Sigma_{6}, \Sigma_{8}, \Sigma_{9}, \Sigma_{10}, \Sigma_{16}$, and $\Sigma_{19}$. The corresponding signified graphs of these $12$ matchings are shown in Figure~\ref{DS5}, where the label of the vertices correspond to that of Figure~\ref{P5}.

\begin{figure}[ht]
\begin{subfigure}{0.24\textwidth}
\begin{tikzpicture}[scale=0.4]
\node[vertex] (v1) at (12,6) {};\node[vertex] (v2) at (15.5,3.5) {};\node[vertex] (v3) at (14,0) {};\node[vertex] (v4) at (10,0) {};\node[vertex] (v5) at (8.5,3.5) {};
\node[vertex] (v6) at (12,4.5) {};\node[vertex] (v7) at (13.8,3.25) {};\node[vertex] (v8) at (13.0,1.4) {};\node[vertex] (v9) at (11,1.4) {};\node[vertex] (v10) at (10.3,3.25) {};
\node [below] at (12,0) {\tiny{$\Sigma_{0}$}};

\foreach \from/\to in {v1/v2,v2/v3,v3/v4,v4/v5,v5/v1,v6/v7,v7/v8,v8/v9,v9/v10,v10/v6,v1/v6,v2/v7,v3/v8,v4/v9,v5/v10}\draw (\from) -- (\to);

\end{tikzpicture}
\end{subfigure}
\hfill
\begin{subfigure}{0.24\textwidth}
\begin{tikzpicture}[scale=0.4]
\node[vertex] (v1) at (12,6) {};\node[vertex] (v2) at (15.5,3.5) {};\node[vertex] (v3) at (14,0) {};\node[vertex] (v4) at (10,0) {};\node[vertex] (v5) at (8.5,3.5) {};
\node[vertex] (v6) at (12,4.5) {};\node[vertex] (v7) at (13.8,3.25) {};\node[vertex] (v8) at (13.0,1.4) {};\node[vertex] (v9) at (11,1.4) {};\node[vertex] (v10) at (10.3,3.25) {};
\node [below] at (12,0) {\tiny{$\Sigma_{1}$}};

\foreach \from/\to in {v2/v3,v3/v4,v4/v5,v5/v1,v6/v7,v7/v8,v8/v9,v9/v10,v10/v6,v1/v6,v2/v7,v3/v8,v4/v9,v5/v10} \draw (\from) -- (\to);
\draw [dotted] (v1) -- (v2);

\end{tikzpicture}
\end{subfigure}
\hfill
\begin{subfigure}{0.24\textwidth}
\begin{tikzpicture}[scale=0.4]
\node[vertex] (v1) at (12,6) {};\node[vertex] (v2) at (15.5,3.5) {};\node[vertex] (v3) at (14,0) {};\node[vertex] (v4) at (10,0) {};\node[vertex] (v5) at (8.5,3.5) {};
\node[vertex] (v6) at (12,4.5) {};\node[vertex] (v7) at (13.8,3.25) {};\node[vertex] (v8) at (13.0,1.4) {};\node[vertex] (v9) at (11,1.4) {};\node[vertex] (v10) at (10.3,3.25) {};
\node [below] at (12,0) {\tiny{$\Sigma_{2}$}};

\foreach \from/\to in {v1/v2,v2/v3,v3/v4,v4/v5,v5/v1,v6/v7,v7/v8,v8/v9,v9/v10,v10/v6,v2/v7,v3/v8,v4/v9,v5/v10} \draw (\from) -- (\to);
\draw [dotted] (v1) -- (v6);

\end{tikzpicture}
\end{subfigure}
\hfill
\begin{subfigure}{0.24\textwidth}
\begin{tikzpicture}[scale=0.4]
\node[vertex] (v1) at (12,6) {};\node[vertex] (v2) at (15.5,3.5) {};\node[vertex] (v3) at (14,0) {};\node[vertex] (v4) at (10,0) {};\node[vertex] (v5) at (8.5,3.5) {};
\node[vertex] (v6) at (12,4.5) {};\node[vertex] (v7) at (13.8,3.25) {};\node[vertex] (v8) at (13.0,1.4) {};\node[vertex] (v9) at (11,1.4) {};\node[vertex] (v10) at (10.3,3.25) {};
\node [below] at (12,0) {\tiny{$\Sigma_{3}$}};

\foreach \from/\to in {v2/v3,v3/v4,v4/v5,v5/v1,v7/v8,v8/v9,v9/v10,v10/v6,v1/v6,v2/v7,v3/v8,v4/v9,v5/v10} \draw (\from) -- (\to);
\draw [dotted] (v1) -- (v2);\draw [dotted] (v6) -- (v7);

\end{tikzpicture}
\end{subfigure}
\hfill
\begin{subfigure}{0.24\textwidth}
\begin{tikzpicture}[scale=0.4]
\node[vertex] (v1) at (12,6) {};\node[vertex] (v2) at (15.5,3.5) {};\node[vertex] (v3) at (14,0) {};\node[vertex] (v4) at (10,0) {};\node[vertex] (v5) at (8.5,3.5) {};
\node[vertex] (v6) at (12,4.5) {};\node[vertex] (v7) at (13.8,3.25) {};\node[vertex] (v8) at (13.0,1.4) {};\node[vertex] (v9) at (11,1.4) {};\node[vertex] (v10) at (10.3,3.25) {};
\node [below] at (12,0) {\tiny{$\Sigma_{4}$}};

\foreach \from/\to in {v2/v3,v3/v4,v4/v5,v5/v1,v6/v7,v8/v9,v9/v10,v10/v6,v1/v6,v2/v7,v3/v8,v4/v9,v5/v10} \draw (\from) -- (\to);
\draw [dotted] (v1) -- (v2);\draw [dotted] (v7) -- (v8);

\end{tikzpicture}
\end{subfigure}
\hfill
\begin{subfigure}{0.24\textwidth}
\begin{tikzpicture}[scale=0.4]
\node[vertex] (v1) at (12,6) {};\node[vertex] (v2) at (15.5,3.5) {};\node[vertex] (v3) at (14,0) {};\node[vertex] (v4) at (10,0) {};\node[vertex] (v5) at (8.5,3.5) {};
\node[vertex] (v6) at (12,4.5) {};\node[vertex] (v7) at (13.8,3.25) {};\node[vertex] (v8) at (13.0,1.4) {};\node[vertex] (v9) at (11,1.4) {};\node[vertex] (v10) at (10.3,3.25) {};
\node [below] at (12,0) {\tiny{$\Sigma_{5}$}};

\foreach \from/\to in {v2/v3,v4/v5,v5/v1,v6/v7,v7/v8,v8/v9,v9/v10,v10/v6,v1/v6,v2/v7,v3/v8,v4/v9,v5/v10} \draw (\from) -- (\to);
\draw [dotted] (v1) -- (v2);\draw [dotted] (v3) -- (v4);

\end{tikzpicture}
\end{subfigure}
\hfill
\begin{subfigure}{0.24\textwidth}
\begin{tikzpicture}[scale=0.4]
\node[vertex] (v1) at (12,6) {};\node[vertex] (v2) at (15.5,3.5) {};\node[vertex] (v3) at (14,0) {};\node[vertex] (v4) at (10,0) {};\node[vertex] (v5) at (8.5,3.5) {};
\node[vertex] (v6) at (12,4.5) {};\node[vertex] (v7) at (13.8,3.25) {};\node[vertex] (v8) at (13.0,1.4) {};\node[vertex] (v9) at (11,1.4) {};\node[vertex] (v10) at (10.3,3.25) {};
\node [below] at (12,0) {\tiny{$\Sigma_{6}$}};

\foreach \from/\to in {v2/v3,v3/v4,v4/v5,v5/v1,v6/v7,v7/v8,v8/v9,v9/v10,v10/v6,v1/v6,v2/v7,v4/v9,v5/v10} \draw (\from) -- (\to);
\draw [dotted] (v1) -- (v2);\draw [dotted] (v3) -- (v8);

\end{tikzpicture}
\end{subfigure}
\hfill
\begin{subfigure}{0.24\textwidth}
\begin{tikzpicture}[scale=0.4]
\node[vertex] (v1) at (12,6) {};\node[vertex] (v2) at (15.5,3.5) {};\node[vertex] (v3) at (14,0) {};\node[vertex] (v4) at (10,0) {};\node[vertex] (v5) at (8.5,3.5) {};
\node[vertex] (v6) at (12,4.5) {};\node[vertex] (v7) at (13.8,3.25) {};\node[vertex] (v8) at (13.0,1.4) {};\node[vertex] (v9) at (11,1.4) {};\node[vertex] (v10) at (10.3,3.25) {};
\node [below] at (12,0) {\tiny{$\Sigma_{8}$}};

\foreach \from/\to in {v2/v3,v3/v4,v4/v5,v5/v1,v6/v7,v7/v8,v9/v10,v10/v6,v1/v6,v2/v7,v3/v8,v4/v9,v5/v10} \draw (\from) -- (\to);
\draw [dotted] (v1) -- (v2);\draw [dotted] (v8) -- (v9);

\end{tikzpicture}
\end{subfigure}
\hfill
\begin{subfigure}{0.24\textwidth}
\begin{tikzpicture}[scale=0.4]
\node[vertex] (v1) at (12,6) {};\node[vertex] (v2) at (15.5,3.5) {};\node[vertex] (v3) at (14,0) {};\node[vertex] (v4) at (10,0) {};\node[vertex] (v5) at (8.5,3.5) {};
\node[vertex] (v6) at (12,4.5) {};\node[vertex] (v7) at (13.8,3.25) {};\node[vertex] (v8) at (13.0,1.4) {};\node[vertex] (v9) at (11,1.4) {};\node[vertex] (v10) at (10.3,3.25) {};
\node [below] at (12,0) {\tiny{$\Sigma_{9}$}};

\foreach \from/\to in {v2/v3,v3/v4,v4/v5,v5/v1,v6/v7,v7/v8,v8/v9,v9/v10,v10/v6,v1/v6,v2/v7,v3/v8,v5/v10} \draw (\from) -- (\to);
\draw [dotted] (v1) -- (v2);\draw [dotted] (v4) -- (v9);

\end{tikzpicture}
\end{subfigure}
\hfill
\begin{subfigure}{0.24\textwidth}
\begin{tikzpicture}[scale=0.4]
\node[vertex] (v1) at (12,6) {};\node[vertex] (v2) at (15.5,3.5) {};\node[vertex] (v3) at (14,0) {};\node[vertex] (v4) at (10,0) {};\node[vertex] (v5) at (8.5,3.5) {};
\node[vertex] (v6) at (12,4.5) {};\node[vertex] (v7) at (13.8,3.25) {};\node[vertex] (v8) at (13.0,1.4) {};\node[vertex] (v9) at (11,1.4) {};\node[vertex] (v10) at (10.3,3.25) {};
\node [below] at (12,0) {\tiny{$\Sigma_{10}$}};

\foreach \from/\to in {v1/v2,v2/v3,v3/v4,v4/v5,v5/v1,v6/v7,v7/v8,v8/v9,v9/v10,v10/v6,v2/v7,v4/v9,v5/v10} \draw (\from) -- (\to);
\draw [dotted] (v1) -- (v6);\draw [dotted] (v3) -- (v8);

\end{tikzpicture}
\end{subfigure}
\hfill
\begin{subfigure}{0.24\textwidth}
\begin{tikzpicture}[scale=0.4]
\node[vertex] (v1) at (12,6) {};\node[vertex] (v2) at (15.5,3.5) {};\node[vertex] (v3) at (14,0) {};\node[vertex] (v4) at (10,0) {};\node[vertex] (v5) at (8.5,3.5) {};
\node[vertex] (v6) at (12,4.5) {};\node[vertex] (v7) at (13.8,3.25) {};\node[vertex] (v8) at (13.0,1.4) {};\node[vertex] (v9) at (11,1.4) {};\node[vertex] (v10) at (10.3,3.25) {};
\node [below] at (12,0) {\tiny{$\Sigma_{16}$}};

\foreach \from/\to in {v1/v2,v7/v8,v3/v4,v4/v5,v5/v1,v6/v7,v4/v9,v9/v10,v10/v6,v2/v7,v3/v8,v5/v10} \draw (\from) -- (\to);
\draw [dotted] (v1) -- (v6);\draw [dotted] (v2) -- (v3);\draw [dotted] (v8) -- (v9);

\end{tikzpicture}
\end{subfigure}
\hfill
\begin{subfigure}{0.24\textwidth}
\begin{tikzpicture}[scale=0.4]
\node[vertex] (v1) at (12,6) {};\node[vertex] (v2) at (15.5,3.5) {};\node[vertex] (v3) at (14,0) {};\node[vertex] (v4) at (10,0) {};\node[vertex] (v5) at (8.5,3.5) {};
\node[vertex] (v6) at (12,4.5) {};\node[vertex] (v7) at (13.8,3.25) {};\node[vertex] (v8) at (13.0,1.4) {};\node[vertex] (v9) at (11,1.4) {};\node[vertex] (v10) at (10.3,3.25) {};
\node [below] at (12,0) {\tiny{$\Sigma_{19}$}};

\foreach \from/\to in {v2/v3,v3/v4,v10/v5,v5/v1,v6/v7,v7/v8,v8/v9,v9/v10,v10/v6,v1/v2,v2/v7,v4/v9} \draw (\from) -- (\to);
\draw [dotted] (v1) -- (v6);\draw [dotted] (v3) -- (v8);\draw [dotted] (v4) -- (v5);

\end{tikzpicture}
\end{subfigure}
\caption{Twelve signed $P(5,1)$.}\label{DS5}
\end{figure}
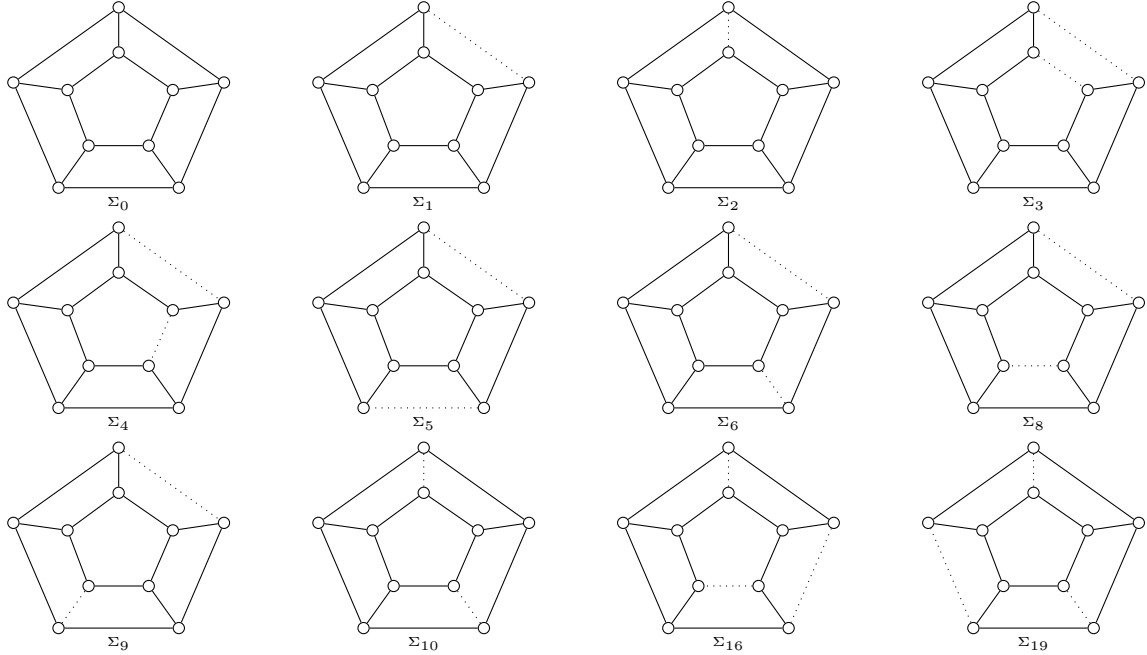

\begin{theorem}
There are exactly twelve signed $P(5,1)$ upto switching isomorphism. 
\end{theorem}
\begin{proof}
The number of negative $4$-cycles, negative $5$-cycles and negative $6$-cycles for the $12$ signed $P(5,1)$ in Figure~\ref{DS5} are given in Table~\ref{table2}. 

\begin{table}[h]
\begin{center}
\begin{tabular}{ |c|c|c|c|c|c|c|c|c|c|c|c|c| } 
 \hline
  & $\Sigma_{0}$ & $\Sigma_{1}$ & $\Sigma_{2}$ & $\Sigma_{3}$ & $\Sigma_{4}$ & $\Sigma_{5}$ & $\Sigma_{6}$ & $\Sigma_{8}$ & $\Sigma_{9}$ & $\Sigma_{10}$ & $\Sigma_{16}$ & $\Sigma_{19}$\\ 
 \hline
number of negative $C_{4}$   & 0 & 1 & 2 & 0 & 2 & 2 & 3 & 2 & 3 & 4 & 4 & 5 \\
 \hline 
number of negative $C_{5}$   & 0 & 1 & 0 & 2 & 2 & 0 & 1 & 2 & 1 & 0 & 2 & 1 \\ 
 \hline 
number of negative $C_{6}$   & 0 & 2 & 2 & 0 & 2 & 4 & 2 & 4 & 4 & 2 & 2 & 0 \\ 
 \hline
 
\end{tabular}
\end{center}
\caption{Number of negative $4$-cycles, $5$-cycles, $6$-cycles of some signed $P(5,1)$.}
\label{table2}
\end{table}
From Theorem~\ref{Signature} and Table~\ref{table2}, it is easy to see that the twelve signed $P(5,1)$, shown in Figure~\ref{DS5}, are non-switching isomorphic. This concludes the proof of the theorem. 
\end{proof}

\section{Signings on $P(7,1)$}\label{P(7,1)}

The graph $P(7,1)$ is shown in Figure~\ref{P7}.
\begin{figure}[ht]
\centering
\begin{tikzpicture}[scale=0.5]
\node[vertex] (v1) at (11.9,7) {};\node[vertex] (v2) at (14.9,5.3) {};\node[vertex] (v3) at (15.3,2.5) {};\node[vertex] (v4) at (13.5,0) {};\node[vertex] (v5) at (10.3,0) {};
\node[vertex] (v6) at (8.7,2.5) {};\node[vertex] (v7) at (9.1,5.3) {};\node[vertex] (v8) at (12.0,5.4) {};\node[vertex] (v9) at (13.5,4.4) {};\node[vertex] (v10) at (13.8,2.8){};
\node[vertex] (v11) at (12.8,1.4) {};\node[vertex] (v12) at (11.2,1.4) {};\node[vertex] (v13) at (10.3,2.8) {};\node[vertex] (v14) at (10.5,4.4) {};

\node [above] at (v1) {$u_{0}$};\node [right] at (v2) {$u_{1}$};\node [right] at (v3) {$u_{2}$};\node [below] at (v4) {$u_{3}$};\node [below] at (v5) {$u_{4}$};
\node [left] at (v6) {$u_{5}$};\node [left] at (v7) {$u_{6}$};\node [below] at (v8) {$v_{0}$};\node at (13.10,4.3) {$v_{1}$};\node [left] at (v10) {$v_{2}$};
\node at (12.6,1.85) {$v_{3}$};\node at (11.5,1.75) {$v_{4}$};\node [right] at (v13) {$v_{5}$};\node at (10.9,4.1) {$v_{6}$};

\foreach \from/\to in {v1/v2,v2/v3,v3/v4,v4/v5,v5/v6,v6/v7,v7/v1,v8/v9,v9/v10,v10/v11,v11/v12,v12/v13,v13/v14,v14/v8,v1/v8,v2/v9,v3/v10,v4/v11,v5/v12,v6/v13,v7/v14} \draw (\from) -- (\to);
\end{tikzpicture}
\caption{The graph $P(7,1)$.}\label{P7}
\end{figure}
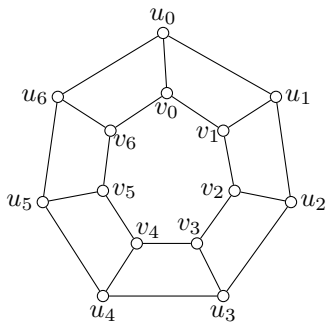
From Theorem \ref{cycle}, we see that the number of 4-cycles, 6-cycles, 7-cycles and 8-cycles in $P(7, 1)$ are $7, 7, 2$ and $7$, respectively. We now find the non-isomorphic matchings of sizes $ 0, 1, 2, 3, 4, 5, 6$ and $7$, upto switching isomorphism. We have the following lemmas to settle the possible cases of matchings of different sizes.

\begin{lemm}\label{forbidden}
Consider the subsets $\sigma_{1} = \{u_{0}u_{1}, v_{0}v_{1}, v_{2}u_{2}\}, \sigma_{2} = \{u_{0}u_{1}, v_{1}v_{2}, u_{2}u_{3}\}, \sigma_{3} = \{u_{0}u_{1}, v_{1}v_{2}, v_{4}v_{5}\}$,
$\sigma_{4} = \{u_{0}u_{1}, v_{0}v_{1}, u_{3}u_{4}\},\sigma_{5} = \{u_{0}u_{1}, v_{0}v_{6}, u_{3}u_{4}\},\sigma_{6} = \{u_{0}v_{0}, u_{1}v_{1}, v_{2}v_{3}\}$ and $\sigma_{7} = \{u_{0}v_{0}, u_{1}v_{1}, u_{2}v_{2}\}$ of edges of $P(7,1)$. If any one of these seven signatures appears in a matching $M_{l}$ of $P(7,1)$, where $l \geq 3$, then $M_{l}$ is switching equivalent to $M_{l^{\prime}}$, where $l^{\prime} \leq l-1$.
\end{lemm}
\begin{proof}
Let $M_{l}^j$ be a matching of size $j$ which contains the set $\sigma_{j}$, where $j=1,\ldots,7$ and $l \geq 3$. Consider the sets $S_1=\{u_{1}, v_{1}, v_{2}\}, S_2=\{u_{1}, v_{2}, u_{2}\}, S_3=\{u_{1}, v_{2}, u_{2}, u_{3}, u_{4}, v_{3}, v_{4}\} , S_4=\{u_{1}, v_{1}, u_{2}, v_{2}, v_{3}, u_{3}\}$, $S_5=\{v_{0}, u_{1}, v_{1}, u_{2}, v_{2}, u_{3}, v_{3}\}, S_6=\{v_{1}, v_{2}, v_{0}\}$ and $S_7=\{u_{0}, u_{1}, u_{2}\}$. If we resign at the vertices belonging to $S_j$, then we get a signature of size upto $l-1$, \textit{i.e.,} $M_{l} \sim M_{l'}$, where $l'\leq l-1$. This proves the lemma.
\end{proof}

 In Lemma~\ref{forbidden}, the inequality $l'\leq (l-1)$ may be strict. For example, if $M_3 = \sigma_4$, then by resigning at the vertices $u_{1}, v_{1}, u_{2}, v_{2}, v_{3}$ and $u_{3}$, we get $M_{3} \sim M_{1}$. The matchings $\sigma_{i}$, where $1 \leq i \leq 7$, are said to be \emph{forbidden} matchings of $P(7,1)$. Let $M_0$ denotes the matching $\Sigma_{1} =\emptyset$ of size zero. Further, there are only two automorphic type matchings of size one, we denote them by $\Sigma_{2}(u_{0}u_{1})$ and $\Sigma_{3}(u_{0}v_{0})$. Any other matching of $P(7,1)$ of size one is automorphic to either $\Sigma_{2}$ or $\Sigma_{3}$, and that $\Sigma_{2}$ is not automorphic to $\Sigma_{3}$.

\begin{lemm} \label{MS7.2}
The number of distinct automorphic type matchings of $P(7,1)$ of size two is $12$.
\end{lemm}
\begin{proof}
 We classify these matchings by looking at the distance of their edges. Recall that the distance between any two edges of $P(7,1)$ is at most four.
\begin{enumerate}

\item[(i)] Let the edges of $M_2$ be at distance two. Five such possible matchings of size two are $\Sigma_{4}(u_{0}u_{1}, v_{0}v_{1}),$ $\Sigma_{5}(u_{0}u_{1}, v_{1}v_{2}),  \Sigma_{6}(u_{0}u_{1}, v_{2}u_{2}), \Sigma_{7}(u_{0}u_{1}, u_{2}u_{3})$ and $\Sigma_{8}(u_{0}v_{0}, u_{1}v_{1})$. Note that any matching of $P(7,1)$ of size two contains either two consecutive spokes, or one spoke and one edge from outer (inner) cycle or two edges from the outer (inner) cycle, or one edge from the inner cycle and one from the outer cycle. Each such possible $M_2$, whose edges are at distance two, is automorphic to one of $\Sigma_8, \Sigma_6, \Sigma_7, \Sigma_4$ and $\Sigma_5$. These five matchings are also pairwise non-automorphic.

\item[(ii)] Let the edges of $M_2$ be at distance three. There are only four automorphic type matchings of size two having edges at distance three. We denote them by $\Sigma_{9}(u_{0}u_{1}, v_{2}v_{3}), \Sigma_{10}(u_{0}u_{1}, v_{3}u_{3}), \Sigma_{11}(u_{0}v_{0}, v_{2}u_{2})$ and $\Sigma_{12}(u_{0}u_{1}, u_{3}u_{4})$. It is easy to see that any other matching of size two whose edges are at distance three is automorphic to one of $\Sigma_{9}, \Sigma_{10}, \Sigma_{11}$ and $\Sigma_{12}$. Further, these matchings are pairwise non-automorphic.

\item[(iii)] Let the edges of $M_2$ be at distance four. There are only three automorphic type matchings of size two whose edges are at distance four. We denote them by $\Sigma_{13}(u_{0}u_{1}, v_{3}v_{4}), \Sigma_{14}(u_{0}u_{1}, v_{4}u_{4})$ and  $\Sigma_{15}(u_{0}v_{0}, v_{3}u_{3})$. Any other $M_2$ whose edges are at distance four is automorphic to one of $\Sigma_{13}, \Sigma_{14}$ and $\Sigma_{15}$. Further, no two of these matchings are automorphic to each other.
\end{enumerate}
This completes the proof.
\end{proof}

\begin{lemm} \label{MS7.3}
The number of distinct automorphic type matchings of $P(7,1)$ of size three is $23$.
\end{lemm}
\begin{proof}
We classify matchings of size three on the basis of the number of spokes contained in it. Since each forbidden matching is a matching of size three and that they are switching equivalent to a matching of size at most two, we consider matchings other than the forbidden matchings.

\begin{enumerate}
\item[(i)] Let $M_3$ have no spoke. The possible distinct  automorphic type matchings of size three without spokes are denoted by $\Sigma_{16}(u_{0}u_{1}, u_{2}u_{3}, u_{4}u_{5}), \Sigma_{17}(u_{0}u_{1}, u_{2}u_{3}, v_{4}v_{5})$ and $\Sigma_{18}(u_{0}u_{1}, u_{4}u_{3}, v_{5}v_{6})$. Any other matching of size three with no spokes is either a forbidden matching or automorphic to one of $\Sigma_{16}, \Sigma_{17}$ and $\Sigma_{18}$. Further, it is easy to see that they are pairwise non-automorphic.

\item[(ii)] Let $M_3$ have only one spoke, say $u_0v_0$. Possible automorphic type matchings are $\Sigma_{19}(u_{0}v_{0}, u_{1}u_{2}, u_{3}u_{4})$, 
$\Sigma_{20}(u_{0}v_{0}, u_{1}u_{2}, u_{4}u_{5}), \Sigma_{21}(u_{0}v_{0}, u_{1}u_{2}, u_{5}u_{6}),  \Sigma_{22}(u_{0}v_{0}, u_{2}u_{3}, u_{4}u_{5}),  
\Sigma_{23}(u_{0}v_{0}, u_{1}u_{2}, v_{2}v_{3})$, \linebreak[4]
$\Sigma_{24}(u_{0}v_{0}, u_{1}u_{2}, v_{3}v_{4}), \Sigma_{25}(u_{0}v_{0}, u_{1}u_{2}, v_{4}v_{5}),  \Sigma_{26}(u_{0}v_{0}, u_{1}u_{2}, v_{5}v_{6}), \Sigma_{27}(u_{0}v_{0}, u_{2}u_{3}, v_{3}v_{4})$,\linebreak[4] $\Sigma_{28}(u_{0}v_{0}, u_{2}u_{3}, v_{4}v_{5}), \Sigma_{29}(u_{0}v_{0}, u_{2}u_{3}, v_{5}v_{6})$ and  $\Sigma_{30}(u_{0}v_{0}, u_{3}u_{4}, v_{5}v_{6})$. Any other matching of size three containing only one spoke is automorphic to one of these twelve matchings. Further, any two of these matchings are pairwise non-automorphic.

\item[(iii)] Let $M_3$ have only two spokes. If the spokes are consecutive then let they be $v_{0}u_{0}$ and $v_{1}u_{1}$. Thus the only possible automorphic type matching is $\Sigma_{31}(v_{0}u_{0}, v_{1}u_{1}, u_{3}u_{4})$. If spokes are at distance three then let the spokes be $v_{0}u_{0}$ and $v_{2}u_{2}$. The possible automorphic type matchings are 
$\Sigma_{32}(v_{0}u_{0}, v_{2}u_{2}, u_{3}u_{4})$ and $\Sigma_{33}(v_{0}u_{0}, v_{2}u_{2}, u_{4}u_{5})$. If the spokes are at distance four then let they be $v_{0}u_{0}$ and $v_{3}u_{3}$. The possible automorphic type matchings are $\Sigma_{34}(v_{0}u_{0}, v_{3}u_{3}, u_{1}u_{2})$ and $\Sigma_{35}(v_{0}u_{0}, v_{3}u_{3}, u_{4}u_{5})$. Because of the forbidden matchings and the automorphism group of $P(7,1)$, it is easy to see that any other matching of size three containing only two spokes is automorphic to one of $\Sigma_{31}, \Sigma_{32}, \Sigma_{33}, \Sigma_{34}$ and $\Sigma_{35}$. Further, these matchings are pairwise non-automorphic.

\item[(iv)] Let $M_3$ have three spokes. The possible automorphic type matchings of size three having three spokes are $\Sigma_{36}(v_{0}u_{0}, v_{1}u_{1}, u_{3}v_{3})$, $\Sigma_{37}(v_{0}u_{0}, v_{1}u_{1}, u_{4}v_{4})$ and $\Sigma_{38}(v_{0}u_{0}, v_{2}u_{2}, u_{4}v_{4})$. Any matching of size three with three spokes is automorphic to $\Sigma_{36}$, $\Sigma_{37}$ or $\Sigma_{38}$. Also, these matchings are pairwise non-automorphic.
\end{enumerate}
This completes the proof of the lemma.
\end{proof}

\begin{lemm} \label{MS7.4}
The number of distinct automorphic type matchings of $P(7,1)$ of size four is $10$.
\end{lemm}
\begin{proof}
We classify the matchings of size four on the basis of number of spokes contained in it. If any $M_4$ contains a forbidden matching then that matching is not considered as a possible candidate for distinct automorphic type matching of size four.
\begin{enumerate}
\item[(i)] Let $M_4$ have no spoke. It is clear that any such $M_4$ has its three edges on outer cycle and the remaining edge on the inner cycle, or two edges on the inner cycle and other two edges on the outer cycle. Therefore, any such $M_4$ must contain one of $\sigma_2, \sigma_3$ and $\sigma_4$. Hence by Lemma~\ref{forbidden}, every matching of size four containing no spoke is equivalent to a matching $M_{l'}$, where $l' \leq 3$.

\item[(ii)] Let $M_4$ have one spoke and let it be $u_0v_0$. It is clear that the remaining three edges of $M_4$ either lie on the outer (inner) cycle, or two edges lie on the outer cycle and one edge lies on the inner cycle. If three edges lie on the outer cycle, then one such matching is $\Sigma_{39}(v_{0}u_{0}, u_{1}u_{2}, u_{3}u_{4}, u_{5}u_{6})$. If two edges lie on the outer cycle and one edge lies on the inner cycle, then the possible automorphic type matchings are $\Sigma_{40}(v_{0}u_{0}, u_{1}u_{2}, u_{3}u_{4}, v_{5}v_{6})$ and $\Sigma_{41}(v_{0}u_{0}, u_{1}u_{2}, v_{3}v_{4}, u_{5}u_{6})$. Any other matching of size four containing only one spoke either contains one of the forbidden matchings or automorphic to one of $\Sigma_{39}, \Sigma_{40}$ and $\Sigma_{41}$. Further, these three matchings are pairwise non-automorphic.

\item[(iii)] Let $M_4$ have two spokes. If the spokes are consecutive then by resigning at the end vertices of the spokes lying on the outer cycle, we find that the matching is equivalent to a matching (signature) of size four. Note that this resultant signature is either equivalent to a matching of size four of case (i) or it is not a matching. Therefore in both cases, it is switching equivalent to a matching $M_{l'}$, where $l' \leq 3$.  If the spokes are at distance three or four, then the possible automorphic type matchings of size four are $\Sigma_{42}(v_{0}u_{0}, v_{2}u_{2}, u_{3}u_{4}, u_{5}u_{6}), \Sigma_{43}(v_{0}u_{0}, v_{2}u_{2}, u_{3}u_{4}, v_{5}v_{6}), \Sigma_{44}(v_{0}u_{0}, v_{2}u_{2}, u_{4}u_{5}, v_{5}v_{6})$,  $\Sigma_{45}(v_{0}u_{0}, v_{3}u_{3}, u_{1}u_{2}, u_{4}u_{5}),  \Sigma_{46}(v_{0}u_{0}, v_{3}u_{3}, u_{1}u_{2}, v_{4}v_{5})$ and $\Sigma_{47}(v_{0}u_{0}, v_{3}u_{3}, u_{4}u_{5}, v_{5}v_{6})$. Any other matching of size four containing only two spokes is automorphic to one of these six matchings. Further, any two of these matchings are pairwise non-automorphic.

\item[(iv)] Let $M_4$ have three spokes. It is clear that a matching of size four cannot contain two or three consecutive spokes upto switchings. Thus the only possible automorphic type matching is \linebreak[4] $\Sigma_{48}(v_{0}u_{0}, v_{2}u_{2}, v_{4}u_{4}, u_{5}u_{6})$.

\item[(v)] Let $M_4$ have four spokes. By resigning at the vertices from the set $\{u_0, u_1, u_2, u_3, u_4, u_4, u_6\}$, we see that such a matching is equivalent to a matching of size three containing three spokes.
\end{enumerate}
This proves the lemma.
\end{proof}

\begin{theorem} \label{M5}
Every matching of $P(7,1)$ of size five is switching equivalent to a matching $M_{l'}$, where $l' \leq 4$.
\end{theorem}
\begin{proof}
We prove the theorem by classifying the matchings of size five on the basis of number of spokes contained in it.
\begin{enumerate}

\item[(i)] Let $M_5$ have no spoke. Note that three edges of $M_5$ may lie on the outer (inner) cycle and remaining
two edges lie on the inner (outer) cycle. It is easy to see that each such combination of five edges of $P(7,1)$ must contain either $\sigma_2$ or $\sigma_4$. Therefore by Lemma~\ref{forbidden}, each such $M_5$ is switching equivalent to a matching $M_{l'}$, where $l' \leq 4$.

\item[(ii)] Let $M_5$ have only one spoke and let it be $u_0v_0$. There are two possibilities for the remaining four edges of $M_5$.\\
(a) Three edges of $M_5$ lie on the outer (inner) cycle and one edge lies on the inner (outer) cycle.\\
(b) Two edges of $M_5$ lie on the outer cycle and remaining two edges lie on the inner cycle.

Note that in (a), the three edges lying on the outer cycle must be $u_1u_2, u_3u_4$ and $u_5u_6$. Therefore for the fifth edge of $M_5$, whichever edge we choose from the inner cycle, $M_5$ will contain either $\sigma_2$ or $\sigma_4$. Hence by Lemma~\ref{forbidden}, each such $M_5$ is switching equivalent to a matching $M_{l^{\prime}}$, where $l^{\prime} \leq 4$.

In (b), let two edges of $M_5$ lying on the outer cycle be  $u_1u_2$ and $u_3u_4$. The possibilities for the remaining two edges from the inner cycle are $\{v_1v_2, v_3v_4\}$, $\{v_1v_2, v_4v_5\}$, $\{v_1v_2, v_5v_6\}$, $\{v_2v_3, v_4v_5\}$, $\{v_2v_3, v_5v_6\}$ and $\{v_3v_4, v_5v_6\}$. In all of these possible cases, it is easy to see that $M_5$ contains a forbidden matching. Hence by Lemma~\ref{forbidden}, each such $M_5$ is switching equivalent to some matching $M_{l^{\prime}}$, where $l^{\prime} \leq 4$. In the similar way, it can be proved that any other matching of size five, containing one spoke and two edges from both outer as well as inner cycles, is switching equivalent to a matching $M_{l^{\prime}}$, where $l^{\prime} \leq 4$.

\item[(iii)] Let $M_5$ have two spokes. If $M_5$ has two consecutive spokes, i.e., spokes at distance two, then by switching at their end vertices lying on outer cycle, we find that $M_5$ is switching equivalent to either a signature of $P(7,1)$ of size five with no spokes or a matching of size five with no spoke. In both cases, $M_5$ is switching equivalent to a matching $M_{l^{\prime}}$, where $l^{\prime} \leq 4$. 

If two spokes of $M_5$ are at distance three then let they be  $u_0v_0$ and $u_2v_2$. The remaining three edges have to be chosen from $\{u_3u_4, u_4u_5, u_5u_6, v_3v_4, v_4v_5, v_5v_6\}$. We can take two edges from outer cycle and one edge from inner cycle or vice-versa. If the two edges from the outer cycle are $u_3u_4$ and $u_5u_6$, then for all possible choices of the edge from the inner cycle, $M_5$ must contain either $\sigma_1$ or $\sigma_2$. Similarly, if two edges are taken from the inner cycle, $M_5$ will contain some forbidden matching. Hence each such $M_5$ is switching equivalent to a matching $M_{l^{\prime}}$, where $l^{\prime} \leq 4$. Similarly, it can be shown that if the two spokes of $M_5$ are at distance four, then also $M_5$ is switching equivalent to a matching $M_{l^{\prime}}$, where $l^{\prime} \leq 4$.

\item[(iv)] Let $M_5$ have three spokes. It is clear from part (iii) that, out of the three spokes, no two spokes \linebreak[4] can be at distance two. Hence the only possibility for such a matching of size five is \linebreak[4] $\{u_0v_0, u_2v_2, u_4v_4, u_5u_6, v_5v_6\}$. But this matching contains $\sigma_1$, hence by Lemma~\ref{forbidden}, we get that this $M_5$ is switching equivalent to a matching $M_{l^{\prime}}$, where $l^{\prime} \leq 4$.

\item[(v)] Let $M_5$ have four spokes. Obviously, resigning at $u_0, u_1, u_2, u_3, u_4 , u_5$ and  $u_6$, each such matching of size five becomes switching equivalent to a matching $M_{l^{\prime}}$, where $l^{\prime} \leq 4$.

\item[(vi)] Let $M_5$ have five spokes. Resigning at $u_0, u_1, u_2, u_3, u_4 , u_5$ and $u_6$, we see that each such matching of size five is switching equivalent to a matching of size two.
\end{enumerate}
This completes the proof of theorem.
\end{proof}

Note that any matching of $P(7,1)$ of size six or seven contains some $M_5$. So by Theorem~\ref{M5}, we get the following corollary.  

\begin{corollary}
Every matching of $P(7,1)$ of size six or seven is switching equivalent to a matching $M_{l^{\prime}}$, where $l^{\prime} \leq 4$.
\end{corollary}

In the preceding lemmas, along with the signatures $\Sigma_{1}, \Sigma_{2}$ and $\Sigma_{3}$, we get $48$ distinct automorphic type matchings of $P(7,1)$ of different sizes. However, among these 48 automorphic type matchings, some of them may be switching isomorphic to each other. Now we attempt to eliminate such switching isomorphic matchings. We have the following observations.

\begin{itemize}
\item In $\Sigma_{8}$, by resigning at $u_{0}, u_{1}$; we get a matching of size two automorphic to $\Sigma_{7}$. Therefore $\Sigma_{8} \sim \Sigma_{7}$.
\item In $\Sigma_{17}$, by resigning at $u_{1}, v_{1}, u_{2}, v_{2}$ and using the automorphism $\gamma$; we get a matching automorphic to $\Sigma_{16}$. Thus $\Sigma_{17} \sim \Sigma_{16}$.
\item In $\Sigma_{18}$, by resigning at $u_{0}, v_{0}, u_{6}, v_{6}$ and using the automorphisms $\delta_3$ followed by $\delta_5$; we get a matching automorphic to $\Sigma_{17}$. Thus $\Sigma_{18} \sim \Sigma_{17}$.
\item In $\Sigma_{20}$, by resigning at $u_{1}, u_{0}$; we get a matching automorphic to $\Sigma_{19}$. Thus $\Sigma_{20} \sim \Sigma_{19}$.
\item In $\Sigma_{21}$, by resigning at $u_{1}, u_{0}, u_{6}$; we get a matching automorphic to $\Sigma_{11}$. Thus $\Sigma_{21} \sim \Sigma_{11}$.
\item In $\Sigma_{25}$, by resigning at $u_{0}, u_{1}$, and using the automorphism $\delta_1$; we get a matching automorphic to $\Sigma_{24}$. Thus $\Sigma_{25} \sim \Sigma_{24}$.
\item In $\Sigma_{26}$, by resigning at $v_{1}, v_{0}, v_{6}, u_{1}$; we get a matching automorphic to $\Sigma_{23}$. Thus $\Sigma_{26} \sim \Sigma_{23}$.
\item In $\Sigma_{29}$, by resigning at $v_{6}, v_{0}$ and using the automorphism $\gamma$; we get a matching automorphic to $\Sigma_{24}$. Thus $\Sigma_{29} \sim \Sigma_{24}$.
\item In $\Sigma_{30}$, by resigning at $v_{6}, v_{0}$ and using the automorphism $\gamma$; we get a matching automorphic to $\Sigma_{25}$. Thus we get $\Sigma_{30} \sim \Sigma_{25}$.
\item In $\Sigma_{31}$, by resigning at $u_{0}, u_{1}$; we get a matching automorphic to $\Sigma_{16}$. Thus $\Sigma_{31} \sim \Sigma_{16}$.
\item In $\Sigma_{34}$, by resigning at $u_{2} u_{3}$; we get a matching automorphic to $\Sigma_{32}$. Thus $\Sigma_{34} \sim \Sigma_{32}$.
\item In $\Sigma_{36}$, by resigning at $u_{0}, u_{1}$ and using the automorphism $\delta_5$; we get a matching automorphic to $\Sigma_{19}$. Thus $\Sigma_{36} \sim \Sigma_{19}$.
\item In $\Sigma_{37}$, by resigning at $u_{0}, u_{1}$; we get a matching automorphic to $\Sigma_{22}$. Thus $\Sigma_{37} \sim \Sigma_{22}$.
\item In $\Sigma_{39}$, by resigning at $u_{6}, u_{0}, u_{1}$; we get a matching automorphic to $\Sigma_{33}$. Thus $\Sigma_{39} \sim \Sigma_{33}$.
\item In $\Sigma_{40}$, by resigning at $u_{0}, u_{1}, u_{4}, u_{5}, u_{6}, v_{4}, v_{5}$ and using $\gamma$; we get a matching automorphic to $\Sigma_{33}$. Thus $\Sigma_{40} \sim \Sigma_{33}$.
\item In $\Sigma_{41}$, by resigning at $u_{0}, u_{1},  u_{6}$ and using $\gamma$; we get a matching automorphic to $\Sigma_{33}$. Thus $\Sigma_{41} \sim \Sigma_{33}$.
\item In $\Sigma_{42}$, by resigning at $u_{0}, u_{1}, u_{2}, u_{3}, u_{6}, v_{4}$ and $v_{5}$; we get matching automorphic to $\Sigma_{38}$. Thus $\Sigma_{42} \sim \Sigma_{38}$.
\item In $\Sigma_{44}$, by resigning at $u_{0}, u_{5}, u_{6}, v_{5}$ and using $\delta_4$; we get a matching automorphic to $\Sigma_{46}$. Thus $\Sigma_{44} \sim \Sigma_{46}$.
\item In $\Sigma_{45}$, by resigning at $u_{0}, u_{1}$; we get a matching automorphic to $\Sigma_{42}$. Thus $\Sigma_{45} \sim \Sigma_{42}$.
\item In $\Sigma_{46}$, by resigning at $u_{2}, u_{3}$ and using $\delta_5$; we get a matching automorphic to $\Sigma_{43}$. Thus $\Sigma_{46} \sim \Sigma_{43}$.
\item In $\Sigma_{47}$, by resigning at $u_{4}, v_{5}, v_{4}, v_{3}$ and using $\gamma$; we get a matching automorphic to $\Sigma_{44}$. Thus $\Sigma_{47} \sim \Sigma_{44}$.
\end{itemize}

Hence we are left with the following matchings: 
$\Sigma_{1}, \Sigma_{2}, \Sigma_{3}, \Sigma_{4}, \Sigma_{5}, \Sigma_{6}, \Sigma_{7}, \Sigma_{8}, \Sigma_{9}, \Sigma_{10}, \Sigma_{11}, \Sigma_{12}, \Sigma_{13}, \Sigma_{14}$, 
$\Sigma_{15}, \Sigma_{18}, \Sigma_{19}, \Sigma_{22}, \Sigma_{23}, \Sigma_{25}, \Sigma_{27}, \Sigma_{32}, \Sigma_{33}, \Sigma_{36}, \Sigma_{37}, \Sigma_{41}$ and 
$\Sigma_{47}$. The corresponding signed graphs of these matchings are shown in Figure~\ref{DS7}, where the label of the vertices correspond to that of Figure~\ref{P7}.

\begin{theorem}
There are exactly $27$ different signed $P(7,1)$ upto switching isomorphism.
\end{theorem}
\begin{proof}
Let $|C^{-}_{4}|, |C^{-}_{6}|, |C^{-}_{7}|$ and $|C^{-}_{8}|$ denote the number of negative $4$-cycles, negative $6$-cycles, negative $7$-cycles and negative $8$-cycles of a signed graph. These numbers for the signed graphs shown in Figure~\ref{DS7} are given in Table~\ref{table3a} and Table~\ref{table3b}. From Table~\ref{table3a}, Table~\ref{table3b} and  Theorem~\ref{Signature}, it is clear that all these 27 signed $P(7,1)$ are non-isomorphic. This concludes the proof of theorem.
\end{proof}

\begin{table}[h]
\begin{center}
\begin{tabular}{ |c|c|c|c|c|c|c|c|c|c|c|c|c|c|c| } 
 \hline
  &  $\Sigma_{1}$ & $\Sigma_{2}$ & $\Sigma_{3}$ & $\Sigma_{4}$ & $\Sigma_{5}$ & $\Sigma_{6}$ & 
 $\Sigma_{7}$ & $\Sigma_{9}$ & $\Sigma_{10}$ & 
 $\Sigma_{11}$ & $\Sigma_{12}$ & $\Sigma_{13}$ & 
 $\Sigma_{14}$ & $\Sigma_{15}$  \\ 
 \hline
$|C^{-}_{4}|$   & 0 & 1 & 2 & 0 & 2 & 3 & 2 & 2 & 3 & 4 & 2 & 2 & 3 & 4 \\
 \hline 
$|C^{-}_{6}|$   & 0 & 2 & 2 & 0 & 2 & 2 & 4 & 4 & 4 & 2 & 4 & 4 & 4 & 4 \\ 
 \hline 
$|C^{-}_{7}|$   & 0 & 1 & 0 & 2 & 2 & 1 & 0 & 2 & 1 & 0 & 0 & 2 & 1 & 0\\ 
 \hline
$|C^{-}_{8}|$  & 0 & 3 & 2 & 0 & 2 & 3 & 4 & 4 & 3 & 4 & 6 & 6 & 5 & 2 \\ 
 \hline
 \end{tabular}
\end{center}
\caption{ The number of negative 4, 6, 7, and 8-cycles of some signed graphs of Figure~\ref{DS7}.}
\label{table3a}
\end{table}
\hfill
\begin{table}[h]
\begin{center}
\begin{tabular}{ |c|c|c|c|c|c|c|c|c|c|c|c|c|c|c| } 
 \hline
  &  $\Sigma_{16}$ & $\Sigma_{19}$ & 
 $\Sigma_{22}$ & $\Sigma_{23}$ & $\Sigma_{24}$ & 
 $\Sigma_{27}$ & $\Sigma_{28}$ & $\Sigma_{32}$ & 
 $\Sigma_{33}$ & $\Sigma_{35}$ & $\Sigma_{38}$ & 
 $\Sigma_{43}$ & $\Sigma_{47}$  \\ 
 \hline
$|C^{-}_{4}|$   & 3 & 4 & 4 & 4 & 4 & 4 & 4 & 5 & 5 & 5 & 6 & 6 & 7 \\
 \hline 
$|C^{-}_{6}|$   & 6 & 4 & 6 & 2 & 4 & 4 & 6 & 2 & 4 & 4 & 2 & 2 & 0 \\ 
 \hline 
$|C^{-}_{7}|$   & 1 & 0 & 0 & 2 & 2 & 2 & 2 & 1 & 1 & 1 & 0 & 2 & 1 \\ 
 \hline
$|C^{-}_{8}|$  & 5 & 4 & 2 & 4 & 4 & 2 & 2 & 5 & 3 & 1 & 4 & 4 & 6 \\ 
 \hline 
\end{tabular}
\end{center}
\caption{The number of negative 4, 6, 7 and 8-cycles of some signed graphs of Figure~\ref{DS7}.}
\label{table3b}
\end{table}
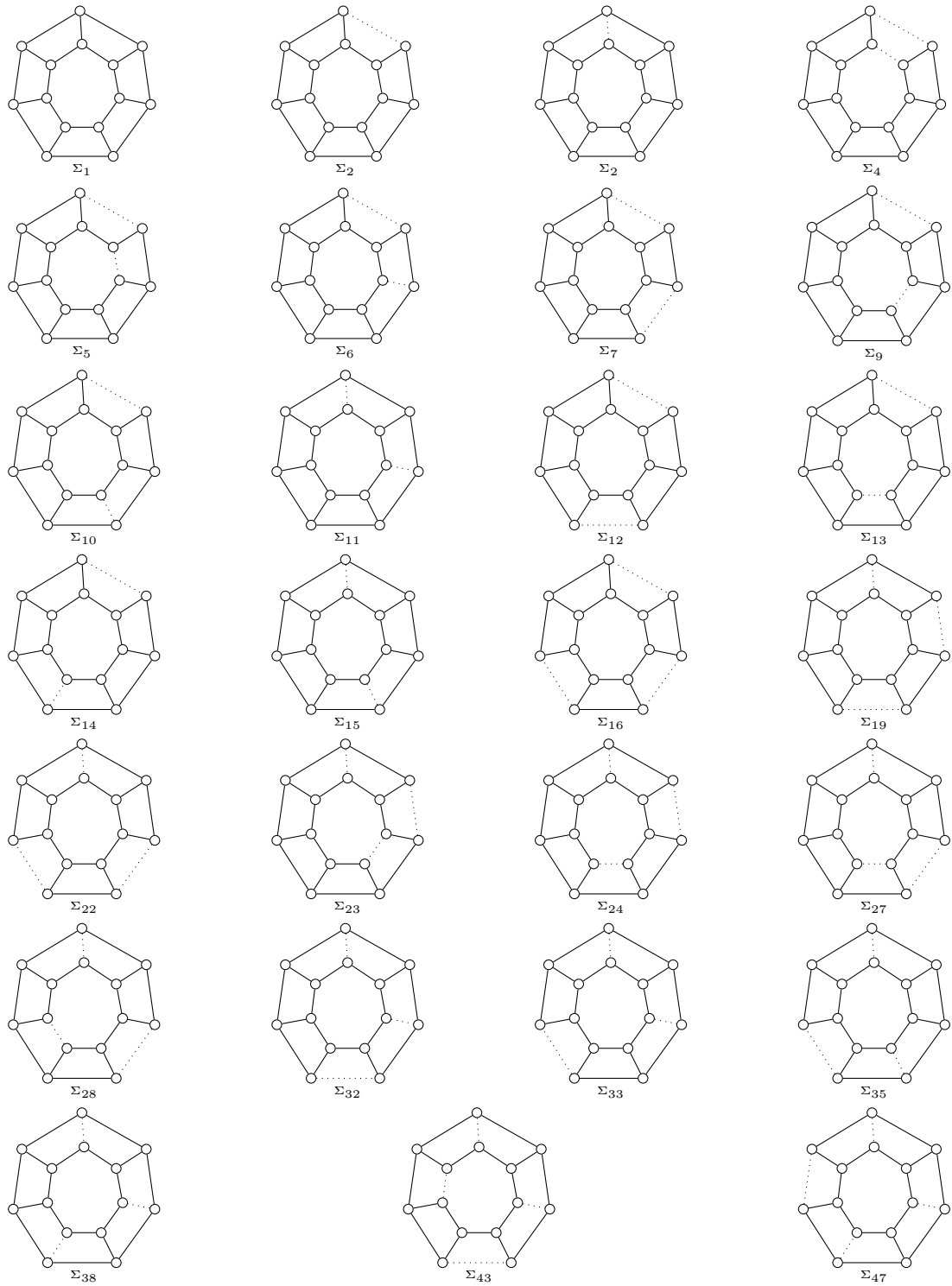
\begin{figure}
\begin{subfigure}{0.24\textwidth}
\begin{tikzpicture}[scale=0.32]
\node[vertex] (v1) at (11.9,7) {};\node[vertex] (v2) at (14.9,5.3) {};\node[vertex] (v3) at (15.3,2.5) {};\node[vertex] (v4) at (13.5,0) {};
\node[vertex] (v5) at (10.3,0) {};\node[vertex] (v6) at (8.7,2.5) {};\node[vertex] (v7) at (9.1,5.3) {};\node[vertex] (v8) at (12.0,5.4) {};
\node[vertex] (v9) at (13.5,4.4) {};\node[vertex] (v10) at (13.8,2.8){};\node[vertex] (v11) at (12.8,1.4) {};\node[vertex] (v12) at (11.2,1.4) {};
\node[vertex] (v13) at (10.3,2.8) {};\node[vertex] (v14) at (10.5,4.4) {};\node [below] at (12,0) {\tiny{$\Sigma_{1}$}};

\foreach \from/\to in {v1/v2,v2/v3,v3/v4,v4/v5,v5/v6,v6/v7,v7/v1,v8/v9,v9/v10,v10/v11,v11/v12,v12/v13,v13/v14,v14/v8,v1/v8,v2/v9,v3/v10,v4/v11,v5/v12,v6/v13,v7/v14} \draw (\from) -- (\to);

\end{tikzpicture}
\end{subfigure}
\hfill
\begin{subfigure}{0.24\textwidth}
\begin{tikzpicture}[scale=0.32]
\node[vertex] (v1) at (11.9,7) {};\node[vertex] (v2) at (14.9,5.3) {};\node[vertex] (v3) at (15.3,2.5) {};\node[vertex] (v4) at (13.5,0) {};
\node[vertex] (v5) at (10.3,0) {};\node[vertex] (v6) at (8.7,2.5) {};\node[vertex] (v7) at (9.1,5.3) {};\node[vertex] (v8) at (12.0,5.4) {};
\node[vertex] (v9) at (13.5,4.4) {};\node[vertex] (v10) at (13.8,2.8){};\node[vertex] (v11) at (12.8,1.4) {};\node[vertex] (v12) at (11.2,1.4) {};
\node[vertex] (v13) at (10.3,2.8) {};\node[vertex] (v14) at (10.5,4.4) {};\node [below] at (12,0) {\tiny{$\Sigma_{2}$}};

\draw [dotted] (v1) -- (v2);
\foreach \from/\to in {v2/v3,v3/v4,v4/v5,v5/v6,v6/v7,v7/v1,v8/v9,v9/v10,v10/v11,v11/v12,v12/v13,v13/v14,v14/v8,v1/v8,v2/v9,v3/v10,v4/v11,v5/v12,v6/v13,v7/v14} \draw (\from) -- (\to);

\end{tikzpicture}
\end{subfigure}
\hfill
\begin{subfigure}{0.24\textwidth}
\begin{tikzpicture}[scale=0.32]
\node[vertex] (v1) at (11.9,7) {};\node[vertex] (v2) at (14.9,5.3) {};\node[vertex] (v3) at (15.3,2.5) {};\node[vertex] (v4) at (13.5,0) {};
\node[vertex] (v5) at (10.3,0) {};\node[vertex] (v6) at (8.7,2.5) {};\node[vertex] (v7) at (9.1,5.3) {};\node[vertex] (v8) at (12.0,5.4) {};
\node[vertex] (v9) at (13.5,4.4) {};\node[vertex] (v10) at (13.8,2.8){};\node[vertex] (v11) at (12.8,1.4) {};\node[vertex] (v12) at (11.2,1.4) {};
\node[vertex] (v13) at (10.3,2.8) {};\node[vertex] (v14) at (10.5,4.4) {};\node [below] at (12,0) {\tiny{$\Sigma_{2}$}};

\draw [dotted] (v1) -- (v8);
\foreach \from/\to in {v1/v2,v2/v3,v3/v4,v4/v5,v5/v6,v6/v7,v7/v1,v8/v9,v9/v10,v10/v11,v11/v12,v12/v13,v13/v14,v14/v8,v2/v9,v3/v10,v4/v11,v5/v12,v6/v13,v7/v14} \draw (\from) -- (\to);

\end{tikzpicture}
\end{subfigure}
\hfill
\begin{subfigure}{0.24\textwidth}
\begin{tikzpicture}[scale=0.32]
\node[vertex] (v1) at (11.9,7) {};\node[vertex] (v2) at (14.9,5.3) {};\node[vertex] (v3) at (15.3,2.5) {};\node[vertex] (v4) at (13.5,0) {};
\node[vertex] (v5) at (10.3,0) {};\node[vertex] (v6) at (8.7,2.5) {};\node[vertex] (v7) at (9.1,5.3) {};\node[vertex] (v8) at (12.0,5.4) {};
\node[vertex] (v9) at (13.5,4.4) {};\node[vertex] (v10) at (13.8,2.8){};\node[vertex] (v11) at (12.8,1.4) {};\node[vertex] (v12) at (11.2,1.4) {};
\node[vertex] (v13) at (10.3,2.8) {};\node[vertex] (v14) at (10.5,4.4) {};\node [below] at (12,0) {\tiny{$\Sigma_{4}$}};

\draw [dotted] (v1) -- (v2);\draw [dotted] (v8) -- (v9);
\foreach \from/\to in {v2/v3,v3/v4,v4/v5,v5/v6,v6/v7,v7/v1,v9/v10,v10/v11,v11/v12,v12/v13,v13/v14,v14/v8,v1/v8,v2/v9,v3/v10,v4/v11,v5/v12,v6/v13,v7/v14} \draw (\from) -- (\to);

\end{tikzpicture}
\end{subfigure}
\hfill
\begin{subfigure}{0.24\textwidth}
\begin{tikzpicture}[scale=0.32]
\node[vertex] (v1) at (11.9,7) {};\node[vertex] (v2) at (14.9,5.3) {};\node[vertex] (v3) at (15.3,2.5) {};\node[vertex] (v4) at (13.5,0) {};
\node[vertex] (v5) at (10.3,0) {};\node[vertex] (v6) at (8.7,2.5) {};\node[vertex] (v7) at (9.1,5.3) {};\node[vertex] (v8) at (12.0,5.4) {};
\node[vertex] (v9) at (13.5,4.4) {};\node[vertex] (v10) at (13.8,2.8){};\node[vertex] (v11) at (12.8,1.4) {};\node[vertex] (v12) at (11.2,1.4) {};
\node[vertex] (v13) at (10.3,2.8) {};\node[vertex] (v14) at (10.5,4.4) {};\node [below] at (12,0) {\tiny{$\Sigma_{5}$}};

\draw [dotted] (v1) -- (v2);\draw [dotted] (v10) -- (v9);
\foreach \from/\to in {v2/v3,v3/v4,v4/v5,v5/v6,v6/v7,v7/v1,v8/v9,v10/v11,v11/v12,v12/v13,v13/v14,v14/v8,v1/v8,v2/v9,v3/v10,v4/v11,v5/v12,v6/v13,v7/v14} \draw (\from) -- (\to);

\end{tikzpicture}
\end{subfigure}
\hfill
\begin{subfigure}{0.24\textwidth}
\begin{tikzpicture}[scale=0.32]
\node[vertex] (v1) at (11.9,7) {};\node[vertex] (v2) at (14.9,5.3) {};\node[vertex] (v3) at (15.3,2.5) {};\node[vertex] (v4) at (13.5,0) {};
\node[vertex] (v5) at (10.3,0) {};\node[vertex] (v6) at (8.7,2.5) {};\node[vertex] (v7) at (9.1,5.3) {};\node[vertex] (v8) at (12.0,5.4) {};
\node[vertex] (v9) at (13.5,4.4) {};\node[vertex] (v10) at (13.8,2.8){};\node[vertex] (v11) at (12.8,1.4) {};\node[vertex] (v12) at (11.2,1.4) {};
\node[vertex] (v13) at (10.3,2.8) {};\node[vertex] (v14) at (10.5,4.4) {};\node [below] at (12,0) {\tiny{$\Sigma_{6}$}};

\draw [dotted] (v1) -- (v2);\draw [dotted] (v10) -- (v3);
\foreach \from/\to in {v2/v3,v3/v4,v4/v5,v5/v6,v6/v7,v7/v1,v8/v9,v9/v10,v10/v11,v11/v12,v12/v13,v13/v14,v14/v8,v1/v8,v2/v9,v4/v11,v5/v12,v6/v13,v7/v14} \draw (\from) -- (\to);

\end{tikzpicture}
\end{subfigure}
\hfill
\begin{subfigure}{0.24\textwidth}
\begin{tikzpicture}[scale=0.32]
\node[vertex] (v1) at (11.9,7) {};\node[vertex] (v2) at (14.9,5.3) {};\node[vertex] (v3) at (15.3,2.5) {};\node[vertex] (v4) at (13.5,0) {};
\node[vertex] (v5) at (10.3,0) {};\node[vertex] (v6) at (8.7,2.5) {};\node[vertex] (v7) at (9.1,5.3) {};\node[vertex] (v8) at (12.0,5.4) {};
\node[vertex] (v9) at (13.5,4.4) {};\node[vertex] (v10) at (13.8,2.8){};\node[vertex] (v11) at (12.8,1.4) {};\node[vertex] (v12) at (11.2,1.4) {};
\node[vertex] (v13) at (10.3,2.8) {};\node[vertex] (v14) at (10.5,4.4) {};\node [below] at (12,0) {\tiny{$\Sigma_{7}$}};

\draw [dotted] (v1) -- (v2);\draw [dotted] (v3) -- (v4);
\foreach \from/\to in {v2/v3,v4/v5,v5/v6,v6/v7,v7/v1,v8/v9,v9/v10,v10/v11,v11/v12,v12/v13,v13/v14,v14/v8,v1/v8,v2/v9,v3/v10,v4/v11,v5/v12,v6/v13,v7/v14} \draw (\from) -- (\to);

\end{tikzpicture}
\end{subfigure}
\hfill
\begin{subfigure}{0.24\textwidth}
\begin{tikzpicture}[scale=0.33]
\node[vertex] (v1) at (11.9,7) {};\node[vertex] (v2) at (14.9,5.3) {};\node[vertex] (v3) at (15.3,2.5) {};\node[vertex] (v4) at (13.5,0) {};
\node[vertex] (v5) at (10.3,0) {};\node[vertex] (v6) at (8.7,2.5) {};\node[vertex] (v7) at (9.1,5.3) {};\node[vertex] (v8) at (12.0,5.4) {};
\node[vertex] (v9) at (13.5,4.4) {};\node[vertex] (v10) at (13.8,2.8){};\node[vertex] (v11) at (12.8,1.4) {};\node[vertex] (v12) at (11.2,1.4) {};
\node[vertex] (v13) at (10.3,2.8) {};\node[vertex] (v14) at (10.5,4.4) {};\node [below] at (12,0) {\tiny{$\Sigma_{9}$}};

\draw [dotted] (v1) -- (v2);\draw [dotted] (v10) -- (v11);
\foreach \from/\to in {v2/v3,v3/v4,v4/v5,v5/v6,v6/v7,v7/v1,v8/v9,v9/v10,v11/v12,v12/v13,v13/v14,v14/v8,v1/v8,v2/v9,v3/v10,v4/v11,v5/v12,v6/v13,v7/v14} \draw (\from) -- (\to);

\end{tikzpicture}
\end{subfigure}
\hfill
\begin{subfigure}{0.24\textwidth}
\begin{tikzpicture}[scale=0.33]
\node[vertex] (v1) at (11.9,7) {};\node[vertex] (v2) at (14.9,5.3) {};\node[vertex] (v3) at (15.3,2.5) {};\node[vertex] (v4) at (13.5,0) {};
\node[vertex] (v5) at (10.3,0) {};\node[vertex] (v6) at (8.7,2.5) {};\node[vertex] (v7) at (9.1,5.3) {};\node[vertex] (v8) at (12.0,5.4) {};
\node[vertex] (v9) at (13.5,4.4) {};\node[vertex] (v10) at (13.8,2.8){};\node[vertex] (v11) at (12.8,1.4) {};\node[vertex] (v12) at (11.2,1.4) {};
\node[vertex] (v13) at (10.3,2.8) {};\node[vertex] (v14) at (10.5,4.4) {};\node [below] at (12,0) {\tiny{$\Sigma_{10}$}};

\draw [dotted] (v1) -- (v2);\draw [dotted] (v4) -- (v11);
\foreach \from/\to in {v2/v3,v3/v4,v4/v5,v5/v6,v6/v7,v7/v1,v8/v9,v9/v10,v10/v11,v11/v12,v12/v13,v13/v14,v14/v8,v1/v8,v2/v9,v3/v10,v5/v12,v6/v13,v7/v14} \draw (\from) -- (\to);

\end{tikzpicture}
\end{subfigure}
\hfill
\begin{subfigure}{0.24\textwidth}
\begin{tikzpicture}[scale=0.33]
\node[vertex] (v1) at (11.9,7) {};\node[vertex] (v2) at (14.9,5.3) {};\node[vertex] (v3) at (15.3,2.5) {};\node[vertex] (v4) at (13.5,0) {};
\node[vertex] (v5) at (10.3,0) {};\node[vertex] (v6) at (8.7,2.5) {};\node[vertex] (v7) at (9.1,5.3) {};\node[vertex] (v8) at (12.0,5.4) {};
\node[vertex] (v9) at (13.5,4.4) {};\node[vertex] (v10) at (13.8,2.8){};\node[vertex] (v11) at (12.8,1.4) {};\node[vertex] (v12) at (11.2,1.4) {};
\node[vertex] (v13) at (10.3,2.8) {};\node[vertex] (v14) at (10.5,4.4) {};\node [below] at (12,0) {\tiny{$\Sigma_{11}$}};

\draw [dotted] (v1) -- (v8);\draw [dotted] (v10) -- (v3);
\foreach \from/\to in {v1/v2,v2/v3,v3/v4,v4/v5,v5/v6,v6/v7,v7/v1,v8/v9,v9/v10,v10/v11,v11/v12,v12/v13,v13/v14,v14/v8,v2/v9,v4/v11,v5/v12,v6/v13,v7/v14} \draw (\from) -- (\to);

\end{tikzpicture}
\end{subfigure}
\hfill
\begin{subfigure}{0.24\textwidth}
\begin{tikzpicture}[scale=0.33]
\node[vertex] (v1) at (11.9,7) {};\node[vertex] (v2) at (14.9,5.3) {};\node[vertex] (v3) at (15.3,2.5) {};\node[vertex] (v4) at (13.5,0) {};
\node[vertex] (v5) at (10.3,0) {};\node[vertex] (v6) at (8.7,2.5) {};\node[vertex] (v7) at (9.1,5.3) {};\node[vertex] (v8) at (12.0,5.4) {};
\node[vertex] (v9) at (13.5,4.4) {};\node[vertex] (v10) at (13.8,2.8){};\node[vertex] (v11) at (12.8,1.4) {};\node[vertex] (v12) at (11.2,1.4) {};
\node[vertex] (v13) at (10.3,2.8) {};\node[vertex] (v14) at (10.5,4.4) {};\node [below] at (12,0) {\tiny{$\Sigma_{12}$}};

\draw [dotted] (v1) -- (v2);\draw [dotted] (v4) -- (v5);
\foreach \from/\to in {v2/v3,v3/v4,v5/v6,v6/v7,v7/v1,v8/v9,v9/v10,v10/v11,v11/v12,v12/v13,v13/v14,v14/v8,v1/v8,v2/v9,v3/v10,v4/v11,v5/v12,v6/v13,v7/v14} \draw (\from) -- (\to);

\end{tikzpicture}
\end{subfigure}
\hfill
\begin{subfigure}{0.24\textwidth}
\begin{tikzpicture}[scale=0.33]
\node[vertex] (v1) at (11.9,7) {};\node[vertex] (v2) at (14.9,5.3) {};\node[vertex] (v3) at (15.3,2.5) {};\node[vertex] (v4) at (13.5,0) {};
\node[vertex] (v5) at (10.3,0) {};\node[vertex] (v6) at (8.7,2.5) {};\node[vertex] (v7) at (9.1,5.3) {};\node[vertex] (v8) at (12.0,5.4) {};
\node[vertex] (v9) at (13.5,4.4) {};\node[vertex] (v10) at (13.8,2.8){};\node[vertex] (v11) at (12.8,1.4) {};\node[vertex] (v12) at (11.2,1.4) {};
\node[vertex] (v13) at (10.3,2.8) {};\node[vertex] (v14) at (10.5,4.4) {};\node [below] at (12,0) {\tiny{$\Sigma_{13}$}};

\draw [dotted] (v1) -- (v2);\draw [dotted] (v11) -- (v12);
\foreach \from/\to in {v2/v3,v3/v4,v4/v5,v5/v6,v6/v7,v7/v1,v8/v9,v9/v10,v10/v11,v13/v12,v13/v14,v14/v8,v1/v8,v2/v9,v3/v10,v4/v11,v5/v12,v6/v13,v7/v14} \draw (\from) -- (\to);

\end{tikzpicture}
\end{subfigure}
\hfill
\begin{subfigure}{0.24\textwidth}
\begin{tikzpicture}[scale=0.33]
\node[vertex] (v1) at (11.9,7) {};\node[vertex] (v2) at (14.9,5.3) {};\node[vertex] (v3) at (15.3,2.5) {};\node[vertex] (v4) at (13.5,0) {};
\node[vertex] (v5) at (10.3,0) {};\node[vertex] (v6) at (8.7,2.5) {};\node[vertex] (v7) at (9.1,5.3) {};\node[vertex] (v8) at (12.0,5.4) {};
\node[vertex] (v9) at (13.5,4.4) {};\node[vertex] (v10) at (13.8,2.8){};\node[vertex] (v11) at (12.8,1.4) {};\node[vertex] (v12) at (11.2,1.4) {};
\node[vertex] (v13) at (10.3,2.8) {};\node[vertex] (v14) at (10.5,4.4) {};\node [below] at (12,0) {\tiny{$\Sigma_{14}$}};

\draw [dotted] (v1) -- (v2);\draw [dotted] (v5) -- (v12);
\foreach \from/\to in {v2/v3,v3/v4,v4/v5,v5/v6,v6/v7,v7/v1,v8/v9,v9/v10,v10/v11,v11/v12,v12/v13,v13/v14,v14/v8,v1/v8,v2/v9,v3/v10,v4/v11,v6/v13,v7/v14} \draw (\from) -- (\to);

\end{tikzpicture}
\end{subfigure}
\hfill
\begin{subfigure}{0.24\textwidth}
\begin{tikzpicture}[scale=0.33]
\node[vertex] (v1) at (11.9,7) {};\node[vertex] (v2) at (14.9,5.3) {};\node[vertex] (v3) at (15.3,2.5) {};\node[vertex] (v4) at (13.5,0) {};
\node[vertex] (v5) at (10.3,0) {};\node[vertex] (v6) at (8.7,2.5) {};\node[vertex] (v7) at (9.1,5.3) {};\node[vertex] (v8) at (12.0,5.4) {};
\node[vertex] (v9) at (13.5,4.4) {};\node[vertex] (v10) at (13.8,2.8){};\node[vertex] (v11) at (12.8,1.4) {};\node[vertex] (v12) at (11.2,1.4) {};
\node[vertex] (v13) at (10.3,2.8) {};\node[vertex] (v14) at (10.5,4.4) {};\node [below] at (12,0) {\tiny{$\Sigma_{15}$}};

\draw [dotted] (v1) -- (v8);\draw [dotted] (v11) -- (v4);
\foreach \from/\to in {v1/v2,v2/v3,v3/v4,v4/v5,v5/v6,v6/v7,v7/v1,v8/v9,v9/v10,v10/v11,v11/v12,v12/v13,v13/v14,v14/v8,v2/v9,v3/v10,v5/v12,v6/v13,v7/v14} \draw (\from) -- (\to);

\end{tikzpicture}
\end{subfigure}
\hfill
\begin{subfigure}{0.24\textwidth}
\begin{tikzpicture}[scale=0.33]
\node[vertex] (v1) at (11.9,7) {};\node[vertex] (v2) at (14.9,5.3) {};\node[vertex] (v3) at (15.3,2.5) {};\node[vertex] (v4) at (13.5,0) {};
\node[vertex] (v5) at (10.3,0) {};\node[vertex] (v6) at (8.7,2.5) {};\node[vertex] (v7) at (9.1,5.3) {};\node[vertex] (v8) at (12.0,5.4) {};
\node[vertex] (v9) at (13.5,4.4) {};\node[vertex] (v10) at (13.8,2.8){};\node[vertex] (v11) at (12.8,1.4) {};\node[vertex] (v12) at (11.2,1.4) {};
\node[vertex] (v13) at (10.3,2.8) {};\node[vertex] (v14) at (10.5,4.4) {};\node [below] at (12,0) {\tiny{$\Sigma_{16}$}};

\draw [dotted] (v1) -- (v2);\draw [dotted] (v3) -- (v4);\draw [dotted] (v5) -- (v6);
\foreach \from/\to in {v2/v3,v4/v5,v6/v7,v7/v1,v8/v9,v9/v10,v10/v11,v11/v12,v12/v13,v13/v14,v14/v8,v1/v8,v2/v9,v3/v10,v4/v11,v5/v12,v6/v13,v7/v14} \draw (\from) -- (\to);

\end{tikzpicture}
\end{subfigure}
\hfill
\begin{subfigure}{0.24\textwidth}
\begin{tikzpicture}[scale=0.33]
\node[vertex] (v1) at (11.9,7) {};\node[vertex] (v2) at (14.9,5.3) {};\node[vertex] (v3) at (15.3,2.5) {};\node[vertex] (v4) at (13.5,0) {};
\node[vertex] (v5) at (10.3,0) {};\node[vertex] (v6) at (8.7,2.5) {};\node[vertex] (v7) at (9.1,5.3) {};\node[vertex] (v8) at (12.0,5.4) {};
\node[vertex] (v9) at (13.5,4.4) {};\node[vertex] (v10) at (13.8,2.8){};\node[vertex] (v11) at (12.8,1.4) {};\node[vertex] (v12) at (11.2,1.4) {};
\node[vertex] (v13) at (10.3,2.8) {};\node[vertex] (v14) at (10.5,4.4) {};\node [below] at (12,0) {\tiny{$\Sigma_{19}$}};

\draw [dotted] (v1) -- (v8);\draw [dotted] (v2) -- (v3);\draw [dotted] (v5) -- (v4);
\foreach \from/\to in {v2/v1,v4/v3,v5/v6,v6/v7,v7/v1,v8/v9,v9/v10,v10/v11,v11/v12,v12/v13,v13/v14,v14/v8,v12/v5,v2/v9,v3/v10,v4/v11,v6/v13,v7/v14} \draw (\from) -- (\to);

\end{tikzpicture}
\end{subfigure}
\hfill
\begin{subfigure}{0.24\textwidth}
\begin{tikzpicture}[scale=0.33]
\node[vertex] (v1) at (11.9,7) {};\node[vertex] (v2) at (14.9,5.3) {};\node[vertex] (v3) at (15.3,2.5) {};\node[vertex] (v4) at (13.5,0) {};
\node[vertex] (v5) at (10.3,0) {};\node[vertex] (v6) at (8.7,2.5) {};\node[vertex] (v7) at (9.1,5.3) {};\node[vertex] (v8) at (12.0,5.4) {};
\node[vertex] (v9) at (13.5,4.4) {};\node[vertex] (v10) at (13.8,2.8){};\node[vertex] (v11) at (12.8,1.4) {};\node[vertex] (v12) at (11.2,1.4) {};
\node[vertex] (v13) at (10.3,2.8) {};\node[vertex] (v14) at (10.5,4.4) {};\node [below] at (12,0) {\tiny{$\Sigma_{22}$}};

\draw [dotted] (v1) -- (v8);\draw [dotted] (v4) -- (v3);\draw [dotted] (v5) -- (v6);
\foreach \from/\to in {v2/v1,v3/v10,v11/v12,v5/v4,v6/v7,v7/v1,v8/v9,v9/v10,v10/v11,v12/v13,v13/v14,v14/v8,v2/v3,v2/v9,v4/v11,v5/v12,v6/v13,v7/v14} \draw (\from) -- (\to);

\end{tikzpicture}
\end{subfigure}
\hfill
\begin{subfigure}{0.24\textwidth}
\begin{tikzpicture}[scale=0.33]
\node[vertex] (v1) at (11.9,7) {};\node[vertex] (v2) at (14.9,5.3) {};\node[vertex] (v3) at (15.3,2.5) {};\node[vertex] (v4) at (13.5,0) {};
\node[vertex] (v5) at (10.3,0) {};\node[vertex] (v6) at (8.7,2.5) {};\node[vertex] (v7) at (9.1,5.3) {};\node[vertex] (v8) at (12.0,5.4) {};
\node[vertex] (v9) at (13.5,4.4) {};\node[vertex] (v10) at (13.8,2.8){};\node[vertex] (v11) at (12.8,1.4) {};\node[vertex] (v12) at (11.2,1.4) {};
\node[vertex] (v13) at (10.3,2.8) {};\node[vertex] (v14) at (10.5,4.4) {};\node [below] at (12,0) {\tiny{$\Sigma_{23}$}};

\draw [dotted] (v1) -- (v8);\draw [dotted] (v2) -- (v3);\draw [dotted] (v10) -- (v11);
\foreach \from/\to in {v2/v1,v4/v3,v5/v6,v6/v7,v7/v1,v8/v9,v9/v10,v4/v5,v11/v12,v12/v13,v13/v14,v14/v8,v12/v5,v2/v9,v3/v10,v4/v11,v6/v13,v7/v14} \draw (\from) -- (\to);

\end{tikzpicture}
\end{subfigure}
\hfill
\begin{subfigure}{0.24\textwidth}
\begin{tikzpicture}[scale=0.33]
\node[vertex] (v1) at (11.9,7) {};\node[vertex] (v2) at (14.9,5.3) {};\node[vertex] (v3) at (15.3,2.5) {};\node[vertex] (v4) at (13.5,0) {};
\node[vertex] (v5) at (10.3,0) {};\node[vertex] (v6) at (8.7,2.5) {};\node[vertex] (v7) at (9.1,5.3) {};\node[vertex] (v8) at (12.0,5.4) {};
\node[vertex] (v9) at (13.5,4.4) {};\node[vertex] (v10) at (13.8,2.8){};\node[vertex] (v11) at (12.8,1.4) {};\node[vertex] (v12) at (11.2,1.4) {};
\node[vertex] (v13) at (10.3,2.8) {};\node[vertex] (v14) at (10.5,4.4) {};\node [below] at (12,0) {\tiny{$\Sigma_{24}$}};

\draw [dotted] (v1) -- (v8);\draw [dotted] (v2) -- (v3);\draw [dotted] (v11) -- (v12);
\foreach \from/\to in {v2/v1,v4/v3,v5/v6,v6/v7,v7/v1,v8/v9,v9/v10,v4/v5,v11/v10,v12/v13,v13/v14,v14/v8,v12/v5,v2/v9,v3/v10,v4/v11,v6/v13,v7/v14} \draw (\from) -- (\to);

\end{tikzpicture}
\end{subfigure}
\hfill
\begin{subfigure}{0.24\textwidth}
\begin{tikzpicture}[scale=0.33]
\node[vertex] (v1) at (11.9,7) {};\node[vertex] (v2) at (14.9,5.3) {};\node[vertex] (v3) at (15.3,2.5) {};\node[vertex] (v4) at (13.5,0) {};
\node[vertex] (v5) at (10.3,0) {};\node[vertex] (v6) at (8.7,2.5) {};\node[vertex] (v7) at (9.1,5.3) {};\node[vertex] (v8) at (12.0,5.4) {};
\node[vertex] (v9) at (13.5,4.4) {};\node[vertex] (v10) at (13.8,2.8){};\node[vertex] (v11) at (12.8,1.4) {};\node[vertex] (v12) at (11.2,1.4) {};
\node[vertex] (v13) at (10.3,2.8) {};\node[vertex] (v14) at (10.5,4.4) {};\node [below] at (12,0) {\tiny{$\Sigma_{27}$}};

\draw [dotted] (v1) -- (v8);\draw [dotted] (v4) -- (v3);\draw [dotted] (v11) -- (v12);
\foreach \from/\to in {v2/v1,v3/v10,v5/v6,v5/v4,v6/v7,v7/v1,v8/v9,v9/v10,v10/v11,v12/v13,v13/v14,v14/v8,v2/v3,v2/v9,v4/v11,v5/v12,v6/v13,v7/v14} \draw (\from) -- (\to);

\end{tikzpicture}
\end{subfigure}
\hfill
\begin{subfigure}{0.24\textwidth}
\begin{tikzpicture}[scale=0.33]
\node[vertex] (v1) at (11.9,7) {};\node[vertex] (v2) at (14.9,5.3) {};\node[vertex] (v3) at (15.3,2.5) {};\node[vertex] (v4) at (13.5,0) {};
\node[vertex] (v5) at (10.3,0) {};\node[vertex] (v6) at (8.7,2.5) {};\node[vertex] (v7) at (9.1,5.3) {};\node[vertex] (v8) at (12.0,5.4) {};
\node[vertex] (v9) at (13.5,4.4) {};\node[vertex] (v10) at (13.8,2.8){};\node[vertex] (v11) at (12.8,1.4) {};\node[vertex] (v12) at (11.2,1.4) {};
\node[vertex] (v13) at (10.3,2.8) {};\node[vertex] (v14) at (10.5,4.4) {};\node [below] at (12,0) {\tiny{$\Sigma_{28}$}};

\draw [dotted] (v1) -- (v8);\draw [dotted] (v12) -- (v13);\draw [dotted] (v3) -- (v4);
\foreach \from/\to in {v2/v1,v3/v10,v5/v6,v5/v4,v6/v7,v7/v1,v8/v9,v9/v10,v10/v11,v12/v11,v13/v14,v14/v8,v2/v3,v2/v9,v4/v11,v5/v12,v6/v13,v7/v14} \draw (\from) -- (\to);

\end{tikzpicture}
\end{subfigure}
\hfill
\begin{subfigure}{0.24\textwidth}
\begin{tikzpicture}[scale=0.33]
\node[vertex] (v1) at (11.9,7) {};\node[vertex] (v2) at (14.9,5.3) {};\node[vertex] (v3) at (15.3,2.5) {};\node[vertex] (v4) at (13.5,0) {};
\node[vertex] (v5) at (10.3,0) {};\node[vertex] (v6) at (8.7,2.5) {};\node[vertex] (v7) at (9.1,5.3) {};\node[vertex] (v8) at (12.0,5.4) {};
\node[vertex] (v9) at (13.5,4.4) {};\node[vertex] (v10) at (13.8,2.8){};\node[vertex] (v11) at (12.8,1.4) {};\node[vertex] (v12) at (11.2,1.4) {};
\node[vertex] (v13) at (10.3,2.8) {};\node[vertex] (v14) at (10.5,4.4) {};\node [below] at (12,0) {\tiny{$\Sigma_{32}$}};

\draw [dotted] (v1) -- (v8);\draw [dotted] (v10) -- (v3);\draw [dotted] (v5) -- (v4);
\foreach \from/\to in {v1/v2,v2/v3,v3/v4,v12/v13,v5/v6,v6/v7,v7/v1,v8/v9,v9/v10,v10/v11,v11/v12,v13/v14,v14/v8,v2/v9,v4/v11,v5/v12,v6/v13,v7/v14} \draw (\from) -- (\to);

\end{tikzpicture}
\end{subfigure}
\hfill
\begin{subfigure}{0.24\textwidth}
\begin{tikzpicture}[scale=0.33]
\node[vertex] (v1) at (11.9,7) {};\node[vertex] (v2) at (14.9,5.3) {};\node[vertex] (v3) at (15.3,2.5) {};\node[vertex] (v4) at (13.5,0) {};
\node[vertex] (v5) at (10.3,0) {};\node[vertex] (v6) at (8.7,2.5) {};\node[vertex] (v7) at (9.1,5.3) {};\node[vertex] (v8) at (12.0,5.4) {};
\node[vertex] (v9) at (13.5,4.4) {};\node[vertex] (v10) at (13.8,2.8){};\node[vertex] (v11) at (12.8,1.4) {};\node[vertex] (v12) at (11.2,1.4) {};
\node[vertex] (v13) at (10.3,2.8) {};\node[vertex] (v14) at (10.5,4.4) {};\node [below] at (12,0) {\tiny{$\Sigma_{33}$}};

\draw [dotted] (v1) -- (v8);\draw [dotted] (v10) -- (v3);\draw [dotted] (v5) -- (v6);
\foreach \from/\to in {v1/v2,v2/v3,v3/v4,v4/v5,v6/v7,v7/v1,v8/v9,v9/v10,v10/v11,v11/v12,v12/v13,v13/v14,v14/v8,v2/v9,v4/v11,v5/v12,v6/v13,v7/v14} \draw (\from) -- (\to);

\end{tikzpicture}
\end{subfigure}
\hfill
\begin{subfigure}{0.24\textwidth}
\begin{tikzpicture}[scale=0.33]
\node[vertex] (v1) at (11.9,7) {};\node[vertex] (v2) at (14.9,5.3) {};\node[vertex] (v3) at (15.3,2.5) {};\node[vertex] (v4) at (13.5,0) {};
\node[vertex] (v5) at (10.3,0) {};\node[vertex] (v6) at (8.7,2.5) {};\node[vertex] (v7) at (9.1,5.3) {};\node[vertex] (v8) at (12.0,5.4) {};
\node[vertex] (v9) at (13.5,4.4) {};\node[vertex] (v10) at (13.8,2.8){};\node[vertex] (v11) at (12.8,1.4) {};\node[vertex] (v12) at (11.2,1.4) {};
\node[vertex] (v13) at (10.3,2.8) {};\node[vertex] (v14) at (10.5,4.4) {};\node [below] at (12,0) {\tiny{$\Sigma_{35}$}};

\draw [dotted] (v1) -- (v8);\draw [dotted] (v5) -- (v6);\draw [dotted] (v11) -- (v4);
\foreach \from/\to in {v1/v2,v2/v3,v3/v4,v4/v5,v5/v12,v6/v7,v7/v1,v8/v9,v9/v10,v10/v11,v11/v12,v12/v13,v13/v14,v14/v8,v2/v9,v3/v10,v6/v13,v7/v14} \draw (\from) -- (\to);

\end{tikzpicture}
\end{subfigure}
\hfill
\begin{subfigure}{0.24\textwidth}
\begin{tikzpicture}[scale=0.33]
\node[vertex] (v1) at (11.9,7) {};\node[vertex] (v2) at (14.9,5.3) {};\node[vertex] (v3) at (15.3,2.5) {};\node[vertex] (v4) at (13.5,0) {};
\node[vertex] (v5) at (10.3,0) {};\node[vertex] (v6) at (8.7,2.5) {};\node[vertex] (v7) at (9.1,5.3) {};\node[vertex] (v8) at (12.0,5.4) {};
\node[vertex] (v9) at (13.5,4.4) {};\node[vertex] (v10) at (13.8,2.8){};\node[vertex] (v11) at (12.8,1.4) {};\node[vertex] (v12) at (11.2,1.4) {};
\node[vertex] (v13) at (10.3,2.8) {};\node[vertex] (v14) at (10.5,4.4) {};\node [below] at (12,0) {\tiny{$\Sigma_{38}$}};

\draw [dotted] (v1) -- (v8);\draw [dotted] (v3) -- (v10);\draw [dotted] (v5) -- (v12);
\foreach \from/\to in {v1/v2,v2/v3,v3/v4,v4/v5,v5/v6,v6/v7,v7/v1,v8/v9,v9/v10,v10/v11,v11/v12,v12/v13,v13/v14,v14/v8,v2/v9,v4/v11,v6/v13,v7/v14} \draw (\from) -- (\to);

\end{tikzpicture}
\end{subfigure}
\hfill
\begin{subfigure}{0.24\textwidth}
\begin{tikzpicture}[scale=0.33]
\node[vertex] (v1) at (11.9,7) {};\node[vertex] (v2) at (14.9,5.3) {};\node[vertex] (v3) at (15.3,2.5) {};\node[vertex] (v4) at (13.5,0) {};
\node[vertex] (v5) at (10.3,0) {};\node[vertex] (v6) at (8.7,2.5) {};\node[vertex] (v7) at (9.1,5.3) {};\node[vertex] (v8) at (12.0,5.4) {};
\node[vertex] (v9) at (13.5,4.4) {};\node[vertex] (v10) at (13.8,2.8){};\node[vertex] (v11) at (12.8,1.4) {};\node[vertex] (v12) at (11.2,1.4) {};
\node[vertex] (v13) at (10.3,2.8) {};\node[vertex] (v14) at (10.5,4.4) {};\node [below] at (12,0) {\tiny{$\Sigma_{43}$}};

\draw [dotted] (v1) -- (v8);\draw [dotted] (v10) -- (v3);\draw [dotted] (v5) -- (v4);\draw [dotted] (v13) -- (v14);
\foreach \from/\to in {v1/v2,v2/v3,v3/v4,v11/v12,v5/v6,v7/v1,v8/v9,v9/v10,v10/v11,v12/v13,v6/v7,v14/v8,v2/v9,v4/v11,v5/v12,v6/v13,v7/v14} \draw (\from) -- (\to);

\end{tikzpicture}
\end{subfigure}
\hfill
\begin{subfigure}{0.24\textwidth}
\begin{tikzpicture}[scale=0.33]
\node[vertex] (v1) at (11.9,7) {};\node[vertex] (v2) at (14.9,5.3) {};\node[vertex] (v3) at (15.3,2.5) {};\node[vertex] (v4) at (13.5,0) {};
\node[vertex] (v5) at (10.3,0) {};\node[vertex] (v6) at (8.7,2.5) {};\node[vertex] (v7) at (9.1,5.3) {};\node[vertex] (v8) at (12.0,5.4) {};
\node[vertex] (v9) at (13.5,4.4) {};\node[vertex] (v10) at (13.8,2.8){};\node[vertex] (v11) at (12.8,1.4) {};\node[vertex] (v12) at (11.2,1.4) {};
\node[vertex] (v13) at (10.3,2.8) {};\node[vertex] (v14) at (10.5,4.4) {};\node [below] at (12,0) {\tiny{$\Sigma_{47}$}};

\draw [dotted] (v1) -- (v8);\draw [dotted] (v10) -- (v3);\draw [dotted] (v6) -- (v7);\draw [dotted] (v5) -- (v12);
\foreach \from/\to in {v1/v2,v2/v3,v3/v4,v4/v5,v5/v6,v13/v14,v7/v1,v8/v9,v9/v10,v10/v11,v11/v12,v12/v13,v14/v8,v2/v9,v4/v11,v6/v13,v7/v14} \draw (\from) -- (\to);

\end{tikzpicture}\end{subfigure}
\caption{Twenty seven signed $P(7,1)$.}\label{DS7}
\end{figure}

\section{Conclusions and Remarks}\label{conclusion}
In the Sections \ref{P(3,1)}, \ref{P(5,1)} and \ref{P(7,1)}, we have found the exact number of non-isomorphic signatures of $P(3,1), P(5,1)$ and $P(7,1)$ upto switching. However this number for the general case is still unknown. So, it is natural to pose the following problem.

\begin{prob}\label{prob1}
What is the exact number of non-switching isomorphic signatures on $P(2n+1,1)$ for all $n \geq 4$ upto switching ?
\end{prob} 

We have a partial answer towards the solution of Problem~\ref{prob1}, that is, we give the answer for non-isomorphic matchings (or minimal signatures) of size two of $P(2n+1,1)$ for all $n \geq 4$. We have the following theorem.

\begin{theorem}
Upto switching isomorphism, the number of matchings of $P(2n+1)$ of size two is $4n-1$, where $n \geq 4$.
\end{theorem}
\begin{proof}
From Theorem~\ref{distance}, we know that edges of any $M_2$ can be at maximum distance $n+1$. Four matchings of size two, whose edges are at distance two are $M_{2}^{21} = \{u_{0}u_{1}, v_{0}v_{1}\}, M_{2}^{22} = \{u_{0}u_{1}, u_{2}v_{2}\}$,  $M_{2}^{23} = \{u_{0}u_{1}, v_{1}v_{2}\}$ and $M_{2}^{24} = \{u_{0}u_{1}, u_{2}u_{3}\}$. For any matching of size two whose edges lie either on the inner cycle or on the outer cycle of $P(2n+1,1)$, there exists some automorphism $\rho_k$ of $P(2n+1,1)$ which maps the matching onto $M_{2}^{24}$. If a matching of size two contains any two spokes which are at distance two, then by resigning at their end points lying on the outer cycle, we see that resultant matching is automorphic to $M_{2}^{24}$. Similarly, it can be shown that all other matchings of size two whose edges are at distance two are automorphic to one of $M_{2}^{21} , M_{2}^{22}$ and $M_{2}^{23}$.

Further, the number of unbalanced $4$-cycles in $M_{2}^{21}, M_{2}^{22}, M_{2}^{23}$ and $M_{2}^{24}$ are $0, 3, 2$ and $2$, respectively. The number of unbalanced $2n+1$-cycles in $M_{2}^{21}, M_{2}^{22}, M_{2}^{23}$ and $M_{2}^{24}$ are $2, 1, 2$ and $0$, respectively. These numbers of unbalanced cycles show that $M_{2}^{21}, M_{2}^{22}, M_{2}^{23}$ and $M_{2}^{24}$ are pairwise non-isomorphic. Hence upto switching isomorphism, $M_{2}^{21}, M_{2}^{22}, M_{2}^{23}$ and $M_{2}^{24}$ are the only matchings of size two whose edges are at distance two. Similarly, For each $i$, one can show upto switching isomorphism that there are exactly four matchings of size two whose edges are at distance $i$, where $3 \leq i < n+1$.

Three matchings of size two whose edges are at distance $n+1$ are  $M_{2}^{(n+1)1} = \{u_{0}u_{1}, u_{n+1}v_{n+1}\}$, $M_{2}^{(n+1)2} = \{u_{0}u_{1}, v_{n}v_{n+1}\}$ and $M_{2}^{(n+1)3} = \{u_{0}v_{0}, u_{n}v_{n}\}$. It is clear that any other matching of size two whose edges are at distance $n+1$ is automorphic to one of $M_{2}^{(n+1)1}, M_{2}^{(n+1)2}$ and $M_{2}^{(n+1)3}$. Further, the number of unbalanced $4$-cycles in $M_{2}^{(n+1)1}, M_{2}^{(n+1)2}$ and $M_{2}^{(n+1)3}$ are $3, 2$ and $4$, respectively. Thus these three matchings are pairwise non-isomorphic.

It is easy to see that any two matchings of size two whose edges are at different distances are different upto switching isomorphisms. This concludes the proof of theorem.
\end{proof}

Having proved this theorem, we propose the following problem.

\begin{prob}\label{prob2}
Can we find the number of non-isomorphic matchings of sizes $3,\ldots,2n+1$ in $P(2n+1,1)$ for all $n \geq 4$?
\end{prob}

In Section~\ref{P(3,1)}, we noticed that $P(3,1)$ has no matching (or minimal signature) of size $3$, upto switching isomorphism. In Section~\ref{P(5,1)}, we noticed that $P(5,1)$ has no matching of sizes $4$ and $5$, upto switching isomorphism. In Section~\ref{P(7,1)}, we noticed that $P(7,1)$ has no matching of sizes $5, 6$ and $7$, upto switching isomorphism. On the basis of these observations, we propose the following conjecture.

\begin{conj}
Any matching in $P(2n+1,1)$ can be of size at most $n+1$, where $n \geq 4$.
\end{conj}
Now Problem~\ref{prob2} can be reformulated as follows.

\begin{prob}\label{prob3}
Can we find the number of non-isomorphic matchings of size $3, 4, 5,\ldots,n+1$ in $P(2n+1,1)$ for all $n \geq 4$ ? 
\end{prob}

\end{document}